\documentclass[12pt]{amsart}

\usepackage{amsmath}
\usepackage{amsthm}
\usepackage{amssymb}
\usepackage{amscd}
\usepackage{bbold}
\usepackage[pdftex]{graphicx}
\usepackage{enumerate}

\newtheorem{lem}{Lemma}[section]
\newtheorem{thm}[lem]{Theorem}
\newtheorem{prop}[lem]{Proposition}

\theoremstyle{definition}
\newtheorem{defn}[lem]{Definition}
\newtheorem{example}[lem]{Example}

\newtheorem{rem}[lem]{Remark}

\newcommand{\CC}{{\mathbb C}}

\newcommand{\RR}{{\mathbb R}}
\newcommand{\Rhat}{\widehat{\mathbb{R}}}

\newcommand{\ZZ}{{\mathbb Z}}

\newcommand{\supp}{\text{supp}\,}       
\def\mbbu{\mathbb{1}}                   
\def\benm{\begin{enumerate}}            
\def\eenm{\end{enumerate}}              
\newcommand{\norm}[1]{\left\Vert #1\right\Vert}         
\newcommand{\inner}[2]{\left\langle #1, #2\right\rangle}    

\title[Generalized Fourier frames in terms of balayage]
{Generalized Fourier frames in terms of balayage}

\date{}

\author{Enrico Au-Yeung}
\address{ Department of Mathematics\\
         University of British Columbia\\
          Vancouver, BC V6T--1Z4\\
         Canada}
 \email{enricoauy@math.ubc.ca}

\author{John J. Benedetto}
\address{Norbert Wiener Center\\
         Department of Mathematics \\
         University of Maryland \\
         College Park, MD 20742 \\
         USA}
\email{jjb@math.umd.edu}
\urladdr{http://www.math.umd.edu/\textasciitilde jjb}

\begin{document}

\begin{abstract}
Based on Beurling's theory of balayage,  we develop the theory of non-uniform sampling in the context
of the theory of frames for the settings of the Short Time Fourier Transform and
pseudo-differential operators. There is sufficient complexity to warrant new examples generally, and to
resurrect the formulation of balayage in terms of covering criteria with an eye towards an expanded
theory as well as computational implementation.

\end{abstract}

\maketitle

\section{Introduction}

\subsection{Background and theme}
There has been a great deal of work during the past quarter century in analyzing, formulating, validating, and extending sampling formulas,
\begin{equation}\label{eq:sampling}
f(x) = \sum f(x_n) s_n,
\end{equation}
for non-uniformly spaced sequences $\{x_n\}$, for specific sequences of sampling functions $s_n$ depending on $x_n$, and for classes of functions $f$ for which such formulas are true. For glimpses into the literature, see the Journal of Sampling Theory in Signal and Image Processing, the influential book by Young \cite{youn2001}, edited volumes such as \cite{BenFer2001}, and specific papers such as those by Jaffard \cite{jaff1991} and Seip \cite{seip1995}. This surge of activity is intimately related to the emergence of wavelet and Gabor theories and more general frame theory. Further, it is firmly buttressed by the profound results of Paley-Wiener \cite{PalWie1934}, Levinson \cite{levi1940}, Duffin-Schaeffer \cite{DufSch1952}, Beurling-Malliavin \cite{BeuMal1962}, \cite{BeuMal1967}, Beurling (unpublished 1959-1960 lectures), and H.~J.~ Landau \cite{land1967}, that themselves have explicit origins by Dini \cite{dini1917}, as well as G. D. Birkhoff (1917), J. L. Walsh (1921), and Wiener (1927), see \cite{PalWie1934}, page 86, for explicit references.

The setting will be in terms of classical spectral criteria to prove non-uniform sampling formulas such as (\ref{eq:sampling}). Our theme is to generalize non-uniform sampling in this setting to the Gabor theory \cite{groc1991}, \cite{FeiSun2006}, \cite{LabWeiWil2004}, as well as to the setting of time-varying signals and pseudo-differential operators. The techniques are based on Beurling's methods from 1959-1960, \cite{beur1966}, \cite{beur1989}, pages 299-315, \cite{beur1989}, pages 341-350, which incorporate balayage, spectral synthesis, and strict multiplicity. Our formulation is in terms of the theory of frames.

\subsection{Definitions}
Let $\mathcal{S}(\mathbb{R}^d)$ be the Schwartz space of rapidly decreasing smooth functions on $d$-dimensional Euclidean space $\RR^d$.  We define the Fourier transform
and inverse Fourier transform
of $f \in \mathcal{S}(\mathbb{R}^d)$ by the formulas,
$$
   \widehat{f}(\gamma) = \int_{\mathbb{R}^d} f(x) e^{-2 \pi i x \cdot \gamma} \ dx
    \quad \text{ and } \quad (\widehat{f})^{\vee}(x) = f(x) = \int_{\mathbb{\widehat{R}}^d} \widehat{f}(\gamma) e^{2 \pi i x \cdot \gamma} \ d\gamma,
$$
respectively.
$\Rhat^d$ denotes $\RR^d$ considered as the spectral domain. If $F \in \mathcal{S}(\Rhat^d)$, then we write $F^\vee(x) = \int_{\Rhat^d}F(\gamma)e^{2\pi i x \cdot \gamma}\,d\gamma$. The notation ``$\int$"' designates integration over $\RR^d$ or $\Rhat^d$.
The Fourier transform extends to tempered distributions. If $X \subseteq \RR^d$, where $X$ is closed, then $M_b(X)$ is the space of bounded Radon measures $\mu$ with support, $\supp(\mu)$, contained in $X$. $C_b(\RR^d)$ denotes the space of complex valued bounded continuous functions on $\RR^d$.
\begin{defn}\label{defn:frame}(Frame)  Let $H$ be a separable Hilbert space.  A sequence $\{x_{n}\}_{n \in \ZZ} \subseteq H$ is a \emph{frame} for $H$ if there are positive constants $A$ and $B$ such that
\[\forall \ f \in H, \quad A \lVert f \rVert^{2} \leq \sum_{n \in \ZZ} |\langle f,x_{n}\rangle|^{2} \leq B \lVert f \rVert^{2} .  \]
The constants $A$ and $B$ are lower and upper frame bounds, respectively.  They are not unique.  We choose $B$ to be the infimum over all upper frame bounds, and we choose $A$ to be the supremum over all lower frame bounds. If $A = B$, we say that the frame is a tight frame or an $A$-tight frame for $H$.
\end{defn}
\begin{defn}(Fourier frame) Let $E \subseteq \mathbb{R}^d$ be a sequence and let $\Lambda \subseteq \widehat{\mathbb{R}}^d$ be a compact set. Notationally, let $e_{x}(\gamma) = e^{2 \pi i x \cdot \gamma}$. The sequence $\mathcal{E}(E) = \{e_{-x}: x \in E \}$ is a \emph{Fourier frame} for $L^2(\Lambda)$ if there are positive constants $A$ and $B$ such that
\[\forall \ F \in L^2(\Lambda), \quad A \lVert F \rVert^{2}_{L^2(\Lambda)} \leq \sum_{x \in E} |\langle F,e_{-x}\rangle|^{2} \leq B \lVert F \rVert^{2}_{L^2(\Lambda)}.  \]
Define the \emph{Paley-Wiener space}, $$PW_{\Lambda}= \{f \in L^2(\RR^d): \supp (\widehat{f}) \subseteq \Lambda\}.$$ Clearly, $\mathcal{E}(E)$ is a Fourier frame for $L^2(\Lambda)$ if and only if the sequence,
$$
\{(e_{-x} \ \mathbb{1}_{\Lambda})^\vee: x \in E \} \subseteq PW_{\Lambda}, 
$$
is a frame for $PW_{\Lambda}$, in which case it is called a \emph{Fourier frame} for $PW_{\Lambda}$. Note that $\inner{F}{e_{-x}} = f(x)$ for $f \in PW_{\Lambda}$, where $\widehat{f} = F \in L^2(\Rhat^d)$ can be  considered an element of $L^2(\Lambda).$
\end{defn}

\begin{rem}
Frames were first defined by Duffin and Schaeffer \cite{DufSch1952}, but appeared explicitly earlier in Paley and Wiener's book \cite{PalWie1934}, page 115. See Christensen's book \cite{chri2003} and Kova\v{c}evi\'{c} and Chebira's articles \cite{KovChe2007a}, \cite{KovChe2007b} for recent expositions of theory and applications. If  $\{x_n\}_{n \in Z} \subseteq H$ is a frame, then there is a topological isomorphism $S : H \longrightarrow \ell^2(Z)$ such that
\begin{equation}
\label{eq:iso}
\forall x \in H, \quad x = \sum_{n \in \ZZ} \inner{x}{S^{-1}(x_n)}x_n = \sum_{n \in \ZZ} \inner{x}{x_n}S^{-1}(x_n).
\end{equation}
Equation (\ref{eq:iso}) illustrates the natural role that frames play in studying non-uniform sampling formulas (\ref{eq:sampling}), see Example \ref{ex:1}.

\end{rem}

Beurling introduced the following definition in his 1959-1960 lectures.
\begin{defn}(Balayage)
Let $E \subseteq \RR^d$ and $\Lambda \subseteq \Rhat^d$ be closed sets. \emph{Balayage} is possible for $(E, \Lambda) \subseteq \RR^d \times \Rhat^d$ if
\begin{equation*}
\forall \mu \in M_b(\RR^d), \mbox{ }\exists \nu \in M_b(E)  \mbox{ such that } \widehat{\mu} = \widehat{\nu} \mbox{ on } \Lambda .
\end{equation*}
\end{defn}

\begin{rem}
{\it a.} The set $\Lambda$ is a collection of group characters in analogy to the Newtonian potential theoretic setting, e.g., \cite{beur1989}, pages 341-350, \cite{land1967}.

{\it b.} The notion of balayage in potential theory is due to Christoffel (1871), e.g., see the remarkable book \cite{ButFeh1981}, edited by Butzer and Feh\'{e}r, and the article therein by Brelot. Then, Poincar\'{e} (1890 and 1899) used the idea of balayage as a method of solving the Dirichlet problem for the Laplace equation. Letting $D \subseteq \RR^d$, $d\geq 3$, be a bounded domain, a balayage or sweeping of the measure $\mu = \delta_y$, $y \in D$, to $\partial D$ is a measure $\nu_y \in M_b(\partial D)$ whose Newtonian potential coincides outside of D with the Newtonian potential of $\delta_y$. In fact, $\nu_y$ is unique and is the harmonic measure on $\partial D$ for $y \in D$, e.g., \cite{kell1929}, \cite{dela1949}.

One then formulates a more general balayage problem: for a given mass distribution $\mu$ inside a closed bounded domain $\overline{D} \subseteq \RR^d$, find a mass distribution $\nu$ on $\partial D$ such that the potentials are equal outside $\overline{D}$ \cite{land1972}, cf. \cite{AdaHed1999}.

Let $\Lambda \subseteq \Rhat^d$ be a closed set. Define $$\mathcal{C}(\Lambda) = \{f \in C_{b}(\RR^d) : \supp(\widehat{f}) \subseteq \Lambda\},$$
cf. the role of $\mathcal{C}(\Lambda)$ in \cite{shap1972}.
\end{rem}

\begin{defn}(Spectral synthesis)
A closed set $\Lambda \subseteq \Rhat^d$ is a set of \emph{spectral synthesis (S-set)} if
\begin{equation}
\label{eq:sset}
\forall f \in \mathcal{C}(\Lambda) \text{ and } \forall \mu \in M_b(\mathbb{R}^d), \quad \widehat{\mu} = 0 \text{ on } \Lambda \Rightarrow \int f \,d\mu = 0,
\end{equation}
see \cite{bene1975}.
\end{defn}

\begin{rem}
{\it a.} The problem of characterizing S-sets emanated from Wiener's Tauberian theorem ideas, and was developed by Beurling in the 1940s. It is ``synthesis'' in that one wishes to approximate $f \in L^{\infty}(\RR^d)$ in the $\sigma(L^\infty (\RR^d), L^1 (\RR^d))$ (weak-$\ast$) topology by finite sums of characters $\gamma: L^\infty (\RR^d) \rightarrow \CC$, where $\gamma$ can be considered an element of $\Rhat^d$ and where $\supp (\delta_\gamma ) \subseteq \supp(\widehat{f})$, which is the so-called spectrum of $f$. Such an approximation is elementary to achieve by convolutions of the measures $\delta_\gamma$, but in this case we lose the essential property that the spectra of the approximants be contained in the spectrum of $f$. It is a fascinating problem whose complete resolution is equivalent to the characterization of the ideal structure of $L^1(\RR^d)$, a veritable Nullstellensatz of harmonic analysis.

{\it b.} We obtain the annihilation property of (\ref{eq:sset}) in the case that $f$ and $\mu$ have balancing smoothness and irregularity. For example, if $\widehat{f} \in D'(\Rhat^d),\,\widehat{\mu} = \phi \in C_c^\infty (\Rhat^d)$, and  $\phi = 0$ on $\supp (\widehat{f})$, then $\widehat{f}(\phi) = 0$, where $\widehat{f}(\phi)$ is sometimes written $\inner{\widehat{f}}{\phi}$. The sphere $S^2 \subseteq \Rhat^3$ is not an S-set (Laurent Schwartz, 1947), and every non-discrete locally compact abelian group $\widehat{G}$, e.g., $\Rhat^d$, contains non-S-sets (Paul Malliavin 1959). On the other hand, polyhedra are S-sets, whereas the 1/3-Cantor set is an S-set with non-S-subsets. We refer to \cite{bene1975} for an exposition of the theory.
\end{rem}

\begin{defn}(Strict multiplicity)
A closed set $\Gamma \subseteq \Rhat^d$ is a set of \emph{strict multiplicity} if
\begin{equation*}
\exists \mu \in M_b(\Gamma)\setminus\{0\} \mbox{ such that } \lim_{\norm{x} \to \infty} |\mu^\vee (x) | = 0.
\end{equation*}
\end{defn}

\begin{rem}
The study of sets of strict multiplicity has its origins in Riemann's theory of sets of uniqueness for trigonometric series, see \cite{bary1964}, \cite{zygm1968}. An early, important, and difficult result is due to Menchov (1916):
\begin{equation*}
\exists \Gamma \subseteq \Rhat / \ZZ \mbox{ and } \exists \mu \in M_b(\Gamma) \setminus \{0\} \mbox{ such that } |\Gamma| = 0 \mbox{ and } \mu^\vee (n) = O((\log |n|)^{-1/2}), |n| \rightarrow \infty.
\end{equation*}
($|\Gamma|$ is the Lebesgue measure of $\Gamma$.) There are refinements of Menchov's result, aimed at increasing the rate of decrease, due to Bary (1927), Littlewood (1936), Salem (1942, 1950), and Iva\v{s}ev-Mucatov (1952, 1956).
\end{rem}

\subsection{Results of Beurling}
The results in this subsection stem from 1959-1960, and the proofs are sometimes sophisticated, see \cite{beur1989}, pages 341-350. Throughout, $E \subseteq \RR^d$ is closed and $\Lambda \subseteq \Rhat^d$ is compact. The following is a consequence of the open mapping theorem.
\begin{prop}\label{prop:110}
 Assume balayage is possible for $(E, \Lambda)$. Then
\begin{equation*}
 \exists K >0 \text{ such that } \forall \mu \in M_b(\RR^d) , \, \inf \{ \norm{\nu}_1 : \nu \in M_b(E) \text{ and } \widehat{\nu} = \widehat{\mu} \text{ on } \Lambda \} \leq K \norm{\mu}_1 .
\end{equation*}
($\norm{\ldots}_1$ designates the total variation norm.)
\end{prop}

The smallest such $K$ is denoted by $K(E, \Lambda)$, and we say that balayage is not possible if $K(E,\Lambda) = \infty$. In fact, \emph{if $\Lambda$ is a set of strict multiplicity, then balayage is possible for $(E,\Lambda)$ if and only if} $K(E, \Lambda) < \infty$, e.g., see Lemma 1 of \cite{beur1989}, pages 341-350. Let $J(E, \Lambda)$ be the smallest $J \geq 0$ such that
\begin{equation*}
 \forall f \in \mathcal{C}(\Lambda) \text{, } \sup_{x \in \RR^d} |f(x)| \leq J \sup_{x \in E} |f(x)|.
\end{equation*}
$J(E, \Lambda)$ could be $\infty$.

The Riesz representation theorem is used to prove the following result. Part \emph{c} is a consequence of parts \emph{a} and \emph{b}.

\begin{prop}
 {\it a.} If $\Lambda$ is a set of strict multiplicity, then $K(E, \Lambda) \leq J(E, \Lambda)$.

 {\it b.} If $\Lambda$ is an S-set, then $J(E,\Lambda) \leq K(E,\Lambda)$.

 {\it c.} Assume that $\Lambda$ is
   an S-set of strict multiplicity and that balayage is possible for $(E, \Lambda)$. If $f \in \mathcal{C}(\Lambda)$ and $f = 0$ on $E$, then $f$ is identically $0$.

\end{prop}

\begin{prop}
 Assume that $\Lambda$ is an S-set of strict multiplicity. Then, balayage is possible for $(E, \Lambda)$ $\Leftrightarrow$
\begin{equation*}
 \exists K(E, \Lambda) > 0 \text{ such that } \forall f \in \mathcal{C}(\Lambda), \quad \norm{f}_{\infty}\leq K(E,\Lambda) \sup_{x \in E}|f(x)|.
\end{equation*}

\end{prop}

The previous results are used in the intricate proof of Theorem \ref{theorem:balayage1}.

\begin{thm}\label{theorem:balayage1}
 Assume that $\Lambda$ is an S-set of strict multiplicity, and that balayage is possible for $(E, \Lambda)$ and therefore $K(E, \Lambda) < \infty$. Let $\Lambda_\epsilon = \{ \gamma \in \Rhat^d: \text{dist}\,(\gamma,\Lambda) \leq \epsilon \}$. Then,
\begin{equation*}
 \exists \, \epsilon_0 > 0 \text{ such that } \forall \, 0 < \epsilon < \epsilon_0 \text{, } K(E,\Lambda_\epsilon) < \infty,
\end{equation*}
i.e., balayage is possible for $(E, \Lambda_\epsilon)$.

\end{thm}

The following result for $\RR^d$ is not explicitly stated in \cite{beur1989}, pages 341-350,
but it goes back to his 1959-1960 lectures, see \cite{wu1998}, Theorem E in \cite{land1967},
Landau's comment on its origins \cite{land2011}, and Example \ref{ex:fouierframebalayage}.
In fact, using Theorem \ref{theorem:balayage1} and Ingham's theorem (Theorem \ref{thm:balayage3}),
Beurling obtained Theorem \ref{theorem:balayage2}.  We have chosen to state Ingham's
theorem (Theorem \ref{thm:balayage3}) in Section 2 as a basic step in the proof of
Theorem \ref{thm:balayage4}, which supposes Theorem \ref{theorem:balayage1} and which
 we chose to highlight as \textit{A fundamental identity of balayage} and in terms of its
quantitative conclusion, (\ref{eq:JB6}) and (\ref{eq:JB7}).  In fact, Theorem \ref{thm:balayage4}
essentially yields Theorem \ref{theorem:balayage2}, see Example \ref{ex:fouierframebalayage}.

\begin{defn}
 A sequence $E \subseteq \RR^d$ is \emph{separated} if $$\exists \, r >0 \text{ such that } \inf \{\norm{x-y}: x, y \in E \text{ and } x \neq y \} \geq r.$$
\end{defn}

\begin{thm}\label{theorem:balayage2}
 Assume that $\Lambda \subseteq \Rhat^d$ is an S-set of strict multiplicity and that
$E \subseteq \RR^d$ is a separated sequence. If balayage is possible for $(E,\Lambda)$,
then $\mathcal{E}(E)$ is a Fourier frame for $L^2(\Lambda)$, i.e., $\{(e_{-x} \
\mathbb{1}_{\Lambda})^\vee: x \in E \}$ is a Fourier frame for $PW_{\Lambda}$.
\end{thm}

\begin{example}\label{ex:1}
The conclusion of Theorem \ref{theorem:balayage2} is the assertion
$$
\forall f \in PW_\Lambda , \quad f = \sum_{x \in E}f(x)S^{-1}(f_x) = \sum_{x \in E}\inner{f}{S^{-1}(f_x)}f_x,
$$
where
$$
f_x(y) = (e_{-x} \ \mathbb{1}_{\Lambda} )^\vee(\gamma)
$$
and
$$
 S(f) = \sum_{x \in E} f(x) (e_{-x} \ \mathbb{1} )^\vee,
$$
cf. (\ref{eq:sampling}) and (\ref{eq:iso}). Clearly, $f_x$ is a type of sinc function. Smooth sampling functions can be introduced into this setup, e.g., Theorem 7.45 of \cite{BenFra1994}, Chapter 7.
\end{example}

\begin{rem}
 Theorem \ref{theorem:balayage2} and results in \cite{beur1966} led to the Beurling covering theorem, see Section \ref{sec:covering}.
\end{rem}

\subsection{Outline}
Now that we have described the background and recalled the required definitions from harmonic
analysis and Beurling's fundamental theorems, we proceed to Section \ref{sec:balayageidentity}, where we state
a basic theorem due to Ingham, as well as what we have called Beurling's fundamental identity of
balayage. This result is a powerful technical tool that we use throughout. 

In Section  \ref{sec:stft}, we prove two theorems,
that are the basis for our frame theoretic non-uniform sampling theory for the Short Time Fourier
Transform (STFT). The second of these theorems, Theorem \ref{thm:stft-frame}, is compared with an
earlier result of Gr{\"o}chenig, that itself goes back to work of Feichtinger and Gr{\"o}chenig.
Section \ref{sec:ex} is devoted to examples that we formulated as avenues for further development
integrating balayage with other theoretical notions. In Section \ref{sec:pdo} we prove the frame inequalities
necessary to provide a non-uniform sampling formula for pseudo-differential operators defined
by a specific class of Kohn-Nirenberg symbols. We view this as the basis for a much broader theory.

Our last mathematical section, Section \ref{sec:covering}, is a brief recollection of Beurling's balayage
results, but formulated in terms of covering criteria and due to a
collaboration of one of the authors in 1990s with Dr. Hui-Chuan Wu. Such coverings in terms of 
polar sets of given band width are a natural vehicle for extending the theory developed herein. Finally, in the Epilogue,
we note the important related contemporary research being conducted in terms of quasicrystals, as well as other 
applications


\section{A fundamental identity of balayage}
\label{sec:balayageidentity}

By construction, and slightly paraphrased, Ingham \cite{ingh1934} proved the following result for the case $d = 1$, see \cite{beur1966}, page 115 for a modification which gives the $d>1$ case. In fact, Beurling gave a version for $d >  1$ in 1953; it is unpublished. In 1962, Kahane \cite{kaha1962} went into depth about the $d > 1$ case. 
\begin{thm}
\label{thm:balayage3}
 Let $\epsilon > 0$ and let $\Omega:[0,\infty) \rightarrow (0, \infty)$ be a continuous function, increasing to infinity. Assume the following conditions:
\begin{equation}
\int_1^\infty \Omega(r) \,\frac{dr}{r^2}<\infty ,
\end{equation}
\begin{equation}
\int exp(-\Omega(\norm{x}))\,dx <\infty ,
\end{equation}
and $\Omega (r) > r^a$ on some interval $[r_0, \infty)$ and for some $a<1$. Then, there is $h\in L^1(\RR^d)$
for which $h(0)=1$, $\supp (\widehat{h})\subseteq \overline{B(0,\epsilon)}$, and
$$|h(x)| = \text{O}(e^{-\Omega \norm{x}}), \quad \norm{x}\rightarrow \infty .$$

\end{thm}

Ingham also proved the converse, which, in fact, requries the Denjoy-Carleman theorem for quasi-analytic functions.

If balayage is possible for $(E,\Lambda)$ and $E \subseteq \RR^d$ is a closed sequence, e.g., if $E$ is separated, then Proposition \ref{prop:110} allows us to write $\widehat{\mu} = \sum_{x \in E} a_x(\mu)\widehat{\delta_x}$ on $\Lambda$, where $\sum_{x \in E}|a_x(\mu)|\leq K(E,\Lambda) \norm{\mu}_1$. In the case $\mu = \delta_y$, we write $a_x(\mu) = a_x(y)$.

We refer to the following result as \emph{A fundamental identity of balayage}.
\begin{thm}
\label{thm:balayage4}
 Let $\Omega$ satisfy the conditions of Ingham's Theorem \ref{thm:balayage3}.
Assume that $\Lambda$ is a compact S-set of strict multiplicity, that $E$ is a
separated sequence, and that balayage is possible for $(E,\Lambda)$.
Choose $\epsilon > 0$ from Beurling's Theorem \ref{theorem:balayage1} so that
$K(E,\Lambda_\epsilon) < \infty$. For this $\epsilon > 0$, take h from Ingham's
Theorem \ref{thm:balayage3}. Then, we have
\begin{equation}\label{eq:JB6}
\forall y \in \RR^d \text{ \rm{and} } \forall f \in \mathcal{C}(\Lambda), \quad f(y) = \sum_{x \in E} f(x) a_x(y) h(x-y),
\end{equation}
where
\begin{equation}\label{eq:JB7}
\sup_{y \in \RR^d} \sum_{x \in E}|a_x(y)| \leq K(E,\Lambda_\epsilon) < \infty .
\end{equation}
In particular, we have
$$\forall y \in \RR^d \text{ \rm{and} }\forall \gamma \in \Lambda , \quad e^{2 \pi i y \cdot \gamma} = \sum_{x \in E}a_x(y)h(x-y)e^{2 \pi i x \cdot \gamma}.
$$

\begin{proof}
Since balayage is possible for $(E,\Lambda_\epsilon)$, we have that $(\delta_y)^\wedge = (\sum_{x \in E}a_x(y)\delta_x)^\wedge$ on $\Lambda_\epsilon$ and that $$\sum_{x \in E}|a_x(y)| \leq K(E, \Lambda_\epsilon)\norm{\delta_y}_1$$ for each $y\in \RR^d$. Thus, (\ref{eq:JB7}) is obtained. Next, for each fixed $y\in \RR^d$, define the measure, $$\eta_y (w) = h_y(w)\left(\delta_y - \sum_{x \in E}a_x(y)\delta_x\right)(w) \in M_b(\RR^d),$$ where $h_y (w) = h(w-y)$. Then, we have
\begin{align*}
(\eta_y)^\wedge(\gamma) &= \left[ (h_y)^\wedge \ast \left(\delta_y - \sum_{x\in E}a_x(y)\delta_x \right)^\wedge \right](\gamma)\\
&= \int \widehat{h}(\gamma - \lambda) e^{-2 \pi i y \cdot (\gamma - \lambda)}
\left(\delta_y - \sum_{x\in E}a_x(y)\delta_x \right)^\wedge (\lambda) \, d \lambda \\
&= \int_{(\Lambda_\epsilon)^c} \widehat{h}(\gamma - \lambda) e^{-2 \pi i y \cdot (\gamma - \lambda)}\left(\delta_y - \sum_{x\in E}a_x(y)\delta_x \right)^\wedge (\lambda) \, d \lambda \\
\end{align*}
on $\Rhat^d$. If $\gamma \in \Lambda$ and $\lambda \in (\Lambda_\epsilon)^c$, then $\widehat{h} (\gamma - \lambda) = 0$. Consequently, we obtain $$\forall y \in \RR^d\text{ and }\forall \gamma \in \Lambda, \quad (\eta_y)^\wedge (\gamma) = 0.$$ Thus, since $\Lambda$ is an S-set and $h(0) = 1$, we obtain (\ref{eq:JB6}) from the definition of $\eta_y$.
\end{proof}

\end{thm}

\begin{example}\label{ex:fouierframebalayage}
 Theorem \ref{thm:balayage4} can be used to prove Beurling's sufficient condition for a Fourier frame in terms of balayage (Theorem \ref{theorem:balayage2}), see part b. For convenience, let $\Lambda$ be symmetric about $0 \in \Rhat^d$, i.e., $-\Lambda = \Lambda$.

 {\it a.} Using the notation of Theorem \ref{thm:balayage4}, we have the following estimate.
\begin{align*}
 \sum_{x \in E} |\int a_x(y)h(x-y)f(y)\,dy|^2 &\leq \sum_{x\in E} \int |a_x(y)||h(x-y)|^2 \,dy \int |a_x(y)||f(y)|^2 \,dy \\
&\leq C \norm{h}_2^2 \int \left(\sum_{x \in E}|a_x(y)|\right)|f(y)|^2\,dy \\
&\leq C \norm{h}_2^2 K(E, \Lambda_\epsilon) \norm{f}_2^2,
\end{align*}
where $C$ is a uniform bound of $\{|a_x(y)|: x \in E, y \in \RR^d \}$.

  {\it b.} It is sufficient to prove the lower frame bound. Let $F \in L^2(\Lambda)$ be considered as an element of $(PW_\Lambda)^\wedge$, i.e., $\widehat{f} = F$ vanishes off of $\Lambda$ and $f \in L^2(\RR^d)$. We shall show that
\begin{equation}\label{eq:JB8}
A \norm{F}_{L^2(\Lambda)} \leq \left(\sum_{x \in E}|f(x)|^2 \right)^{1/2},
\end{equation} where $A$ is independent of $F\in L^2(\Lambda)$.
\begin{align*}
 \norm{F}_{L^2(\Lambda)}^2 &= \int_{\Lambda}\overline{F(\lambda)}\left(\int f(y) e^{-2 \pi i y \cdot \lambda}\,dy\right)\,d\lambda \\
& = \int_{\Lambda}\overline{F(\lambda)}\left(\int f(y) \left(\sum_{x \in E}a_x(y) h(x-y)e^{-2 \pi i x \cdot \lambda}\right)\,dy \right)\,d\lambda \\
&= \sum_{x \in E}\overline{f(x)}\left(\int a_x(y)h(x-y)f(y)\,dy\right) \\
&\leq \left(\sum_{x \in E} |f(x)|^2\right)^{1/2}\left(\sum_{x\in E} \left|\int a_x(y)h(x-y)f(y)\, dy \right|^2\right)^{1/2} \\
&\leq \left[ C \norm{h}_2^2 K(E,\Lambda_\epsilon)\right]^{1/2}\left(\sum_{x\in E} |f(x)|^2 \right)^{1/2}\norm{f}_2,
\end{align*}
and so we set $A = 1 / [C\norm{h}_2^2 K(E, \Lambda_\epsilon)]^{1/2}$ to obtain (\ref{eq:JB8}).

\end{example}


\section{Short time Fourier transform (STFT) frame inequalities}
\label{sec:stft}

\begin{defn}
\emph{a.}  Let $f, g \in L^2(\mathbb{R}^{d})$.
 The {\it short-time Fourier transform} (STFT) of $f$ with respect to $g$ is the function $V_{g}f$ on $\mathbb{R}^{2d}$ defined as
\[ \quad V_{g}f(x, \omega) = \int f(t) \overline{g(t-x)} \ e^{- 2 \pi i t \cdot \omega } \ dt,\]
\end{defn}
\noindent see \cite{groc2001}, \cite{groc2006}.\\

\emph{b.}  The STFT is uniformly continuous on $\mathbb{R}^{2d}$.  Further, for a fixed
``window'' $g \in L^{2}(\mathbb{R}^{d})$ with $\|g\|_{2} = 1$,
we can recover the original function $f \in L^{2}(\mathbb{R}^{d})$ from its STFT $V_{g}f$
by means of the vector-valued integral inversion formula,
\begin{equation}
\label{eq:InversionSTFT}
f = \int \int  V_{g}f(x, \omega) \ e_{\omega} \tau_{x} g \ d\omega \ dx,
\end{equation}
where modulation $e_{\omega}$ was defined earlier and translation $\tau_x$ is
defined as $\tau_{x}g(t) = g(t-x)$.  Explicitly, Equation (\ref{eq:InversionSTFT}) signifies that we have the vector-valued mapping, $(x,\omega) \mapsto e_{\omega} \tau_{x} g \in L^{2}(\mathbb{R}^{d})$, and
\[\forall \ h \in L^{2}({\mathbb R}^d), \ \langle f, h \rangle = \int \int \left[ \int V_{g}f(x, \omega) (e_{\omega} \tau_{x} g(t) ) \overline{h(t)} \ dt \right] d\omega dx.\]
Also, if $\widehat{f} = F$ and $\widehat{g} = G$, where $f, g \in L^{2}(\mathbb{R}^d)$, then one obtains the {\it fundamental identity of time frequency analysis,}
\begin{equation}
\label{eq:time-frequency}
V_{g}f(x,\omega) = e^{-2 \pi i x \cdot \omega } V_{G}F(\omega,-x).
\end{equation}

\emph{c.}  Let $g_0(x) = 2^{d/4} e^{- \pi \| x \|^2 }.$  Then $G_0(\gamma) = \widehat{g}_0(\gamma) = 2^{d/4} e^{- \pi \| \gamma \|^2 }$ and $\| g_0 \|_2 = 1$, see \cite{BenCza2009} for properties of $g_0.$  The {\it Feichtinger algebra}, ${\mathcal S}_0({\mathbb R}^d),$ is
 \[ \mathcal{S}_0(\mathbb{R}^d) = \{ f \in L^2(\mathbb{R}^d) \colon \| f \|_{\mathcal{S}_0} = \| V_{g_0}f \|_1 < \infty \}. \]
For now it is useful to note that the Fourier transform of $\mathcal{S}_0(\mathbb{R}^d)$ is an isometric isomorphism onto itself, and, in particular, $f \in \mathcal{S}_0(\mathbb{R}^d)$ if and only if $F \in \mathcal{S}_0(\widehat{\mathbb{R}}^d)$.

\begin{thm}
\label{thm:1NonUniform}
Let $E = \{x_n\}\subseteq \RR^d$ be a
separated sequence, that is symmetric about $0 \in \mathbb{R}^{d}$; and let $\Lambda \subseteq \widehat{\mathbb{R}}^{d}$ be an S-set of strict multiplicity, that is compact, convex, and symmetric about $0 \in \Rhat^d$.  Assume balayage is possible for $(E, \Lambda)$.  Further, let $g \in L^2(\mathbb{R}^{d}), \,\widehat{g} = G,$ have the property that  $\norm{g}_2 = 1$.

{\it a.}  We have that
$$
\exists \ A > 0, \quad \text{such that } \quad \forall f \in PW_\Lambda \backslash \{0\}, \quad \widehat{f} = F,
$$
\begin{equation}
\label{eq:z}
A \| f \|_{2}^{2} = A \|F \|_{2}^{2}  \leq   \sum_{x \in E} \int | V_{G} F(\omega, x)|^{2} \ d\omega =  \sum_{x \in E} \int | V_{g}f(x, \omega)|^2 \ d\omega.
\end{equation}

{\it b.}  Let $g \in \mathcal{S}_0(\mathbb{R}^d)$.  We have that
$$ 
     \exists \ B > 0, \quad \text{such that } \quad \forall f \in PW_\Lambda \backslash \{0\}, \quad \widehat{f} = F,
$$
\begin{equation}
\label{eq:zz}
  \sum_{x \in E} \int | V_{g} f(x, \omega)|^{2} \ d\omega =  \sum_{x \in E} \int | V_{G}F(\omega, -x)|^2 \ d\omega \leq B \| F \|_2^2 = B \| f \|_2^2,
\end{equation}
where $B$ can be taken as $2^{d/2}\ C \| V_{g_0}g \|_{1}^2$ and where
$$
      C = {\rm sup}_{u \in {\mathbb R}^d} \sum_{x \in E} e^{-\|x-u\|^2}.
$$
see the technique in \cite{FeiZimm1998}, Lemma 3.2.15, cf. \cite{FeiSun2007}, Lemma 3.2.

\begin{proof}
{\it a.i.} We first combine the $STFT$ and balayage to compute
\begin{eqnarray}
\label{eqn:eqn4balayage}
   & & \| f \|_{2}^{2} = \int_{\Lambda} F(\gamma) \ \overline{F(\gamma) } \ d \gamma \\
   & = & \int_{\Lambda} F(\gamma) \ \left( \int \int \overline{ V_{G}F(y, \omega)} \ \overline{\ e_{\omega}(\gamma)} \ \overline{ G(\gamma - y)} \ d\omega \ dy \right) \ d \gamma  \nonumber \\
   & = & \int_{\Lambda} F(\gamma) \ \left( \int \int \overline{ V_{G}F(y, \omega)}  \ \overline{ G(\gamma - y)}  \left( \ \sum_{x \in E} \overline{a_{x}(\omega)}  \ \overline{h(x - \omega)} \ e^{ -2 \pi i x \cdot \gamma} \right) \ d\omega \ dy \right) \ d \gamma \nonumber  \\
   & = & \int \int \overline{ V_{G}F(y, \omega)} \ \left( \ \sum_{x \in E} \overline{a_{x}(\omega)}  \ \overline{h(x
 - \omega)}  \  \int F(\gamma) \  \overline{G(\gamma - y)} \ e^{- 2 \pi i x \cdot \gamma}  \ d \gamma \right) \ d\omega \ dy \nonumber  \\
   & = & \int \int \overline{ V_{G}F(y, \omega)} \ \left( \ \sum_{x \in E} \overline{a_{x}(\omega)}  \ \overline{h(x - \omega)} \ V_{G}F(y, x
) \right) d \omega \ dy \nonumber \\
& = & \int \left[ \ \sum_{x \in E} \left( \int \overline{ V_{G}F(y, \omega)} \ \overline{a_{x}(\omega)} \ \overline{ h(x - \omega)}   \ d \omega \right) \ V_{G}F(y, x)  \right] \ dy \nonumber \\
& \leq & \int \left( \sum_{x \in E} \left| \int a_{x}(\omega) \ h(x - \omega) \  V_{G}F(y, \omega) \ d \omega \right|^{2} \right)^{1/2} \left(\sum_{x \in E} \left| V_{G}F(y, x) \right|^{2} \right)^{1/2} \ dy . \nonumber
\end{eqnarray}

{\it a.ii.} We shall show that there is a constant $C > 0,$ independent of $f \in PW_{\Lambda}$,
such that
\begin{equation}
\label{eq:eqn5cauchy}
\forall \ y \in \mathbb{R}^{d}, \  \sum_{x \in E} \left| \int a_{x}(\omega) \ h(x - \omega) V_{G}F(y,
\omega) \ d \omega \right|^{2}
\ \leq  C^2 \int \left| V_{G}F(y, \omega) \right|^{2} \ d\omega.
\end{equation}
The left side of (\ref{eq:eqn5cauchy}) is bounded above by
\begin{eqnarray*}
& & \sum_{x \in E} \left( \int | a_{x}(\omega)| \ |h(x - \omega)|^{2} \ d\omega \right) \left(\int | a_{x}(\omega)| \ |V_{G}F(y, \omega)|^{2} \ d\omega  \right) \\
& \leq & \sum_{x \in E} \left(K_{1} \ \int |h(x - \omega)|^{2} \ d\omega \right) \left(\int | a_{x}(\omega)| \ |V_{G}F(y, \omega)|^{2} \ d\omega  \right)  \\
& = & K_{1} \  \| h \|_2^2 \ \sum_{x \in E} \int | a_{x}(\omega)|  \  |V_{G}F(y, \omega)|^{2} \ d\omega     \\
& = & K_{1} \ \| h \|_2^2  \ \int \left(\sum_{x \in E} | a_{x}(\omega)| \right)  |V_{G}F(y, \omega)|^{2} \ d\omega \\
& \leq & K_1 \ K_2 \  \| h \|_2^2 \ \int |V_{G}F(y, \omega)|^{2} \ d\omega,
\end{eqnarray*}
where we began by using H\"{o}lder's inequality and where $K_1$ and $K_2$ exist
because of (\ref{eq:JB7}) in Theorem \ref{thm:balayage4}.  Let $C^2 = K_1 K_2 \ \| h \|_2^2$.

{\it a.iii.} Combining parts \emph{a.i} and \emph{a.ii}, we have from
(\ref{eqn:eqn4balayage}) and (\ref{eq:eqn5cauchy}) that
\begin{eqnarray*}
& & \| f \|_2^{2} = \int_{\Lambda} F(\gamma) \ \overline{F(\gamma) } \ d \gamma \\
& \leq & \int \ C \ \left(\int |V_{G}F(y, \omega) |^{2} \ d\omega  \right)^{1/2} \ \left(\sum_{x \in E} |V_{G}F(y, x) |^{2} \right)^{1/2} \ dy \\
& \leq &  C \ \left(\int \int |V_{G}F(y, \omega) |^{2} \ d\omega \ dy \right)^{1/2} \ \left(\int \sum_{x \in E} |V_{G}F(y, x)|^{2} \ dy \right)^{1/2} \\
& = & C \ \left( \int_{\Lambda} |F(\gamma)|^2 \ d\gamma \right)^{1/2} \ \left(\int \sum_{x \in E} |V_{G}F(y, x)|^{2} \ dy \right)^{1/2},
\end{eqnarray*}
where we have used H\"{o}lder's inequality and the fact that the STFT is an isometry from $L^2(\mathbb{R}^d)$ into $L^2(\mathbb{R}^{2d})$.
Consequently, by the symmetry of $E$, we have
\begin{eqnarray*}
\frac{1}{C^2} \| f \|_2^{2} & = & \frac{1}{C^2} \ \int_{\Lambda} |F(\gamma)|^2 \ d\gamma  \\
& \leq & \int  \sum_{x \in E} | V_{G} F(\omega, -x)|^{2} \ d\omega = \int \sum_{x \in E} | V_{g}f(x, \omega)|^2 d \omega,
\end{eqnarray*}
where we have used ($\ref{eq:time-frequency}$).  Part \emph{a} is completed by setting $A = 1/C^2.$

{\it b.i.}  The proof of (\ref{eq:zz}) will require the reproducing formula \cite{FeiSun2007}, page 412:
\begin{equation}
\label{eq:FeiSun}
V_{g}f(y, \gamma) = \langle V_{g_0}f, V_{g_0}( e_{\gamma}\tau_{y}g) \rangle,
\end{equation}
where $\widehat{g}_0 = G_0.$ Equation (\ref{eq:FeiSun}) is a consequence of the inversion formula,
$$
f = \int \int V_{g_0}f(x, \omega) e_{\omega} \tau_x g_0 \ d\omega \ dx,
$$
and substituting the right side into the definition $\langle f, e_{\gamma} \tau_y g \rangle$ of $V_gf(y, \gamma).$  Equation (\ref{eq:FeiSun}) is valid
for all $f, g \in L^2(\mathbb{R}^d).$

{\it b.ii.} Using Equation (\ref{eq:FeiSun}) from part {\it b.i} we compute
$$
     \sum_{x \in E} \int | V_{g} f(x, \omega)|^{2} d\omega
$$
$$
       = \int \sum_{x \in E} |\langle V_{g_0}f, V_{g_0}( e_{\omega}\tau_{x}g) \rangle|^2 d\omega
$$
$$
     =    \int \sum_{x \in E} |\int \int \overline{V_{g_0}f(y,\gamma)}\ V_{g_0}( e_{\omega}\tau_{x}g)(y,\gamma)
     \ dy\ d\gamma|^2 d\omega
$$
$$
      \leq \int \sum_{x \in E}\left(\left(\int \int|V_{g_0}f(y,\gamma)|^2|V_{g_0}(e_{\omega}\tau_{x}g)(y,\gamma)|
      \ dy\ d\gamma\right)\left(\int \int |V_{g_0}(e_{\omega}\tau_{x}g)(y,\gamma)|\ dy\ d\gamma\right)\right) d\omega.
$$
{\it b.iii.} Since
$$
     V_{g_0}(e_{\omega}\tau_{x})g)(y,\gamma) = \int g(t-x)\ \overline{g_{0}(t-y)}\ e^{-2\pi i t\cdot (\gamma - \omega)}dt
$$
$$
     = e^{-2\pi i x\cdot (\gamma - \omega)}\ \int g(u)\ \overline{g_{0}(u+(x-y))}\ e^{-2\pi i u\cdot (\gamma - \omega)}du,
$$
we have
$$
    |V_{g_0}(e_{\omega}\tau_{x})g)(y,\gamma)| \leq |V_{g_0}g(y-x, \gamma - \omega).
$$
Inserting this inequality into the last term of part {\it b.ii}, the inequality of part {\it b.ii} becomes
$$
      \sum_{x \in E} \int | V_{g} f(x, \omega)|^{2} \ d\omega
$$
$$
      \leq  \int \sum_{x \in E}\left(\left(\int \int|V_{g_0}f(y,\gamma)|^2|V_{g_0}g(y-x,\gamma -\omega)|
      \ dy\ d\gamma \right) 
    \left(\int \int |V_{g_0}g)(y-x,\gamma - \omega)|\ dy\ d\gamma \right)\right) d\omega
$$
$$
      = \|V_{g_0}g\|_1\ \int \sum_{x \in E}\left(\int \int|V_{g_0}f(y,\gamma)|^2|V_{g_0}g(y-x,\gamma -\omega)|
      \ dy\ d\gamma\right)\ d\omega
$$
$$
      \leq \|V_{g_0}g\|_1\ \int \int|V_{g_0}f(y,\gamma)|^2\ \left(\int \sum_{x \in E}
     |V_{g_0}g(y-x,\gamma -\omega)|\ d\omega\right)\ dy\ d\gamma.
$$
{\it b.iv.} By the reproducing formula, Equation (\ref{eq:FeiSun}), the integral-sum
factor in the last term of part {\it b.iii} is
$$
     \int \sum_{x \in E}
     |V_{g_0}g(y-x,\gamma -\omega)|\ d\omega
$$
$$
   = \int \sum_{x \in E} |\int\int\ V_{g_0}g(z,\zeta)\ \overline{V_{g_0}(e_{\gamma - \omega}\tau_{y - x}g_0)(z,\zeta)}\
    dz\ d\zeta|d\omega
$$
$$
     = \int \sum_{x \in E} |\int\int\ V_{g_0}g(z,\zeta)\ \left(\overline{\int g_{0}(u)\overline{g_{0}(u-(z+x-y))}
     \ e^{-2\pi i u\cdot (\zeta - \gamma + \omega)}\ du}\right) dz\ d\zeta|\ d\omega
$$
$$
     = \int \sum_{x \in E} |\int\int\ V_{g_0}g(z,\zeta)\ \overline{V_{g_0}g_{0}(z+(x-y), \zeta + (\omega - \gamma))}\ dz\ d\zeta|d\omega
$$
$$
    \leq \int \int |V_{g_0}g(z,\zeta)|\ \left(\int \sum_{x \in E} |V_{g_0}g_{0}(z+(x-y),\zeta +(\omega - \gamma))|
    d\omega\right)\ dz\ d\zeta.
$$
{\it b.v.} Substituting the last term of part {\it b.iv} in the last term of part {\it b.iii}, the inequality
of part {\it b.ii} becomes
$$
       \sum_{x \in E} \int | V_{g} f(x, \omega)|^{2} \ d\omega
$$
$$
   \leq  \|V_{g_0}g\|_1 \int\int|V_{g_0}f(y,\gamma)|^2 \times
$$
$$
   \left(\int\int|V_{g_0}g(z,\zeta)|
   \left(\sum_{x \in E}\left(\int|V_{g_0}g_0(z+(x-y)),\zeta + (\omega - \gamma))| d\omega
   \right)\right) dz\ d\zeta\right)dy\ d\gamma
$$
$$
   = \|V_{g_0}g\|_1 \int\int|V_{g_0}f(y,\gamma)|^2\left(\int\int |V_{g_0}g(z,\zeta)|\left(\sum_{x \in E} K(x,y,z,\gamma,\zeta)\right)dz\ 
   d\zeta\right)dy\ d\gamma,
$$
where
$$
     K(x,y,z,\gamma,\zeta) = e^{-\frac{\pi}{2}\|z+(x-y)\|^2}\ \int e^{-\frac{\pi}{2}\|\zeta + (\omega - \gamma\|^2}d\omega.
$$
Hence,
$$
        \sum_{x \in E} \int | V_{g} f(x, \omega)|^{2} \ d\omega
     \leq 2^{\frac{d}{2}}\ C\ \|V_{g_0}g\|_{1}^2\ \|V_{g_0}f\|^2,
$$
where
$$
           C = {\rm sup}_{u \in {\mathbb R}^d} \sum_{x \in E} e^{-\|x-u\|^2}.
$$
The fact, $C < \infty$, is straightforward to verify, but see \cite{NarWar1991} and \cite{NarSivWar1994},
Lemma 2.1,
for an insightful, refined estimate of $C.$ The proof of part {\it b} is completed by a simple application of
Equation (\ref{eq:STFT_orthogonality}).
\end{proof}
\end{thm}


We now recall a special case of a fundamental theorem of Gr{\"o}chenig for non-uniform Gabor frames, see \cite{groc1991}, Theorem S, and \cite{groc2001}, Theorem 13.1.1, cf. \cite{FeiGro1988} and \cite{FeiGro1989} for a precursor of this result,
presented in an almost perfectly disguised way for the senior author to understand.
The general case of Gr{\"o}chenig's theorem is true for the class of modulation spaces,
$M_{v}^{1}({\mathbb R}^d),$ where the Feichtinger algebra, ${\mathcal S}_{0}({\mathbb R}^d),$
is the case that the weight $v$ is identically $1$ on ${\mathbb R}^d.$ The author's proof at all levels of
generalization involves a significant analysis of convolution operators on the
Heisenberg group. See \cite{groc2001} for
an authoritative exposition of modulation spaces as well as their history.


\begin{thm}
\label{thm:groch}
Given any $g \in  \mathcal{S}_0(\mathbb{R}^d)$. There is $r = r(g) > 0$ such that if
$E = \{(s_n, \sigma_n)\} \subseteq {\mathbb R}^d \times {\widehat{\mathbb R}}^d$ is a separated sequence with the property that
$$
   \bigcup_{n=1}^{\infty} \overline{B((s_n,{\sigma}_n),r(g))} ={\mathbb R}^d \times {\widehat{\mathbb R}}^d,
$$
then the frame operator, $S = S_{g,E},$ defined by
$$
   S_{g,E}\,f = {\sum}_{n=1}^{\infty}\langle f, {\tau}_{s_n}e_{\sigma_n}g\rangle\, {\tau}_{s_n}e_{\sigma_n}g,
$$
is invertible on $\mathcal{S}_0(\mathbb{R}^d)$.

Moreover, every $f \in  \mathcal{S}_0(\mathbb{R}^d)$ has a non-uniform Gabor expansion,
$$
f = {\sum}_{n=1}^{\infty} \langle f, \tau_{s_n} e_{\sigma_n} g \rangle S_{g,E}^{-1}(\tau_{s_n} e_{\sigma_n}g),
$$
where the series converges unconditionally in $ \mathcal{S}_0(\mathbb{R}^d)$.

($E$ depends on $g.$)

\end{thm}

The following result can be compared with Theorem \ref{thm:groch}.

\begin{thm}
\label{thm:stft-frame}
Let $E = \{(s_{n}, \sigma_{n})\} \subseteq \mathbb{R}^d \times  \widehat{\mathbb{R}}^d $ be a separated sequence; and
let $\Lambda \subseteq \widehat{\mathbb{R}}^d \times \mathbb{R}^{d}$ be an S-set of strict multiplicity that is compact, convex, and symmetric about $0 \in \widehat{\mathbb{R}}^d \times \mathbb{R}^{d}.$
Assume balayage is possible for $(E, \Lambda)$.
Further, let $g \in L^2(\mathbb{R}^{d}),\, \widehat{g} = G,$ have the property that  $\norm{g}_2 = 1$.
We have that
$$
 \exists \ A, \ B > 0, \quad \text{ such that } \quad \forall f \in
 \mathcal{S}_0({\mathbb R}^d), \quad \text{for which } \quad {\rm supp}(\widehat{V_gf}) \subseteq \Lambda,
$$
\begin{equation}
\label{eq:stft-frame}
A \norm{f}_2^2  \leq  \, {\sum}_{n=1}^{\infty}  | V_{g} f(s_n, \sigma_n)|^{2}  \leq B \norm{f}_2^2.
\end{equation}
Consequently, the frame operator, $S = S_{g,E},$ is invertible in $L^2({\mathbb R}^d)$--norm
on the subspace of $\mathcal{S}_0(\mathcal{R}^d),$ whose elements $f$ have the property,
$supp\,(\widehat{V_gf}) \subseteq \Lambda.$

Moreover, every $f \in  \mathcal{S}_0(\mathbb{R}^d)$ satisfying the support
condition, ${\rm supp}(\widehat{V_gf}) \subseteq \Lambda,$ has a non-uniform Gabor expansion,
$$
f = {\sum}_{n=1}^{\infty} \langle f, \tau_{s_n} e_{\sigma_n} g \rangle S_{g,E}^{-1}(\tau_{s_n} e_{\sigma_n}g),
$$
where the series converges unconditionally in $L^2(\mathbb{R}^d)$.

($E$ does not depend on $g.$)

\begin{proof}

{\it a.}  Using Theorem \ref{thm:balayage4} for the setting $ \mathbb{R}^d \times  \widehat{\mathbb{R}}^d$, where $h \in L^1(\mathbb{R}^d \times \widehat{\mathbb{R}}^d)$ from Ingham's theorem has the property that ${\rm supp}(\widehat{h}) \subseteq \overline{B(0, \epsilon)} \subseteq \widehat{\mathbb{R}}^d \times \mathbb{R}^{d},$ we compute
\begin{equation}
\label{eq:moyal-beurl}
 \int | f(x) |^2 \ dx = \int \int | V_gf(y, \omega) |^2 \ dy \ d\omega
\end{equation}
$$
    = \int \int \overline{V_{g}f(y, \omega)} {\sum}_{n=1}^{\infty} a_{s_n,{\sigma}_n}(y, \omega) 
    h(s_n - y, \sigma_n - \omega) V_gf(s_n, \sigma_n) \ dy \ d\omega,
$$
where
$$
    V_gf(y, \omega) =   {\sum}_{n=1}^{\infty} a_{s_n,{\sigma}_n}(y, \omega) 
    h(s_n - y, \sigma_n - \omega) V_gf(s_n, \sigma_n) 
$$
and
$$
        {\rm sup}_{(y,\omega) \in {\mathbb R}^d \times \widehat{\mathbb R}^d}\,
        \sum_{n=1}^{\infty}  |a_{s_n,{\sigma}_n}(y, \omega)| \leq K(E,{\Lambda}_\epsilon) < \infty.
$$  
Interchanging summation and integration on the right side of Equation (\ref{eq:moyal-beurl}),
we use H{\"o}lder's inequality to obtain
$$
    \int | f(x) |^2 \ dx \leq
$$
\begin{equation}
\label{eq:prodVbeurl}
   \left(\sum_{n=1}^{\infty} |V_{g}f(s_n,{\sigma}_n)|^2\right)^{1/2}\ \left(\sum_{n=1}^{\infty}|\int\int
    a_{s_n,{\sigma}_n}(y,\omega)
   h(s_n - y, {\sigma}_n - \omega)\ \overline{V_{g}f(y,\omega)}\ dy\ d{\omega}|^2\right)^{1/2}
\end{equation}
$$
     \leq {S_1}^{1/2}\ {S_2}^{1/2}.
$$

We bound the second sum $S_2$ using H{\"o}lder's inequality for the integrand,
$$
   [(a_{s_n,{\sigma}_n}(y,\omega))^{1/2}h(s_n - y, {\sigma}_n - \omega)] [(a_{s_n,{\sigma}_n}(y,\omega))^{1/2}
        \overline{V_{g}f(y,\omega)}],
$$
 as follows:  
 $$
   S_2 \leq \sum_{n=1}^{\infty} \left(\int \int|a_{s_n,{\sigma}_n}(y,\omega)||h(s_n - y, {\sigma}_n - \omega)|^2\,dy\ d{\omega}\ \int \int|a_{s_n,{\sigma}_n}(y,\omega)||V_{g}f(y,\omega)|^2\,dy\,d{\omega}\right)
$$
\begin{equation}
\label{eq:S2ineq}
   \leq K_1\ \sum_{n=1}^{\infty} \left(\int\int |h(s_n - y, {\sigma}_n -\omega)|^2\,dy\,d{\omega}\,\int\int
   |a_{s_n,{\sigma}_n}(y,\omega)||V_{g}f(y,\omega)|^2\,dy\,d{\omega}\right)
\end{equation}
$$
   = K_1\ \norm{h}_2^2 \int\int \left(\sum_{n=1}^{\infty} |a_{s_n,{\sigma}_n}(y,\omega)||V_{g}f(y,\omega)|^2\right)\ dy\ d{\omega}\:\leq\: K_1K_2 \norm{h}_2^2\norm{f}_2^2,
$$
where $K_1$ is a uniform bound on $\{a_{s_n,{\sigma}_n}(y,\omega)\},$ $K_2$ invokes the full power of
Theorem \ref{thm:balayage4}, and $\norm{f}_2^2 = \norm{V_gf}_2^2.$

Combining (\ref{eq:prodVbeurl}) and (\ref{eq:S2ineq}), we obtain
$$
   \norm{f}_2^2 \leq (S_1K_1K_2)^{1/2}\norm{h}_2\norm{f}_2,
$$
and so the left hand inequality of (\ref{eq:stft-frame}) is valid for $1/(K_1K_2\norm{h}_2^2).$

{\it b.} The right hand inequality of (\ref{eq:stft-frame}) follows directly from the P{\'o}lya-Plancherel theorem, cf. Theorem \ref{thm:1NonUniform}{\it b}.

\end{proof}

\end{thm}


\begin{example}
\label{ex:suppFT}
{\it a.} In comparing Theorem \ref{thm:groch} with Theorem \ref{thm:stft-frame} a possible weakness of the former is the
dependence of $E$ on $g,$ whereas a possible weakness of the latter is the hypothesis that
${\rm supp}(\widehat{V_{g}f}) \subseteq \Lambda.$ We now show that this latter constraint is of
no major consequence.

Let $f,g \in L^{1}({\mathbb R}^d) \cap L^{2}({\mathbb R}^d).$ We know that $V_{g}f \in L^{2}({\mathbb R}^d \times
\widehat{{\mathbb R}}^d),$ and
$$
    \widehat{V_{g}f}(\zeta,z) = \int \int \left(\int f(t)\ g(t-x)\ e^{-2 \pi i t \cdot \omega}\ dt\right)\ e^{-2 \pi i(x \cdot \zeta + z \cdot \omega)}\ dx\ d{\omega}.
$$
The right side is
$$
   \int \int f(t)\ \left(\int g(t-x)\ e^{-2 \pi i x \cdot \zeta}\ dx\right)\ e^{-2 \pi i t\cdot \omega}
   \ e^{-2 \pi i  z \cdot \omega}\ dt\ d{\omega},
$$
where the interchange in integration follows from the Fubini-Tonelli theorem and the hypothesis that
$f,g \in L^{1}({\mathbb R}^d).$ This, in turn, is
$$
     \hat{g}(-\zeta)\ \int \left(\int f(t)\ e^{-2 \pi i t\cdot \zeta}\ e^{-2 \pi i t\cdot \omega}\ dt\right)\ e^{-2 \pi i z \cdot \omega}\ d{\omega}
$$
$$
   = \hat{g}(-\zeta)\ \int \hat{f}(\zeta + \omega)\ e^{-2 \pi i  z \cdot \omega}\ d{\omega} = e^{-2 \pi i z \cdot \zeta}\ f(-z)\ \hat{g}(-\zeta).
$$
Consequently, we have shown that if $f,g \in L^{1}({\mathbb R}^d) \cap L^{2}({\mathbb R}^d),$ then
\begin{equation}
\label{eqn:suppFT}
  f,g \in L^{1}({\mathbb R}^d) \cap L^{2}({\mathbb R}^d), \quad
   \widehat{V_{g}f}(\zeta,z) =  e^{-2 \pi i z \cdot \zeta}\ f(-z)\ \hat{g}(-\zeta).
\end{equation}

Let $d=1$ and let $\Lambda = [-\Omega,\Omega] \times [-T,T] \subseteq \widehat{{\mathbb R}}^d \times
{\mathbb R}^d.$ We can choose $g \in PW_{[-\Omega,\Omega]},$ where $\hat{g}$ is even and smooth enough so that
$g \in L^{1}(\mathbb R).$ For this window $g,$ we take any even $f \in L^{2}(\mathbb R)$ which is supported in
$[-T,T].$ Equation (\ref{eqn:suppFT}) applies.

{\it b.} Theorems \ref{thm:groch} and \ref{thm:stft-frame} give non-uniform Gabor frame expansions. Generally,
for $g \in L^2({\mathbb R})$, if $\{e_{\sigma_n}{\tau_{s_n}}g\}$ is a frame for $L^2({\mathbb R}),$ then
$E = \{s_n,\sigma_n\} \subseteq {\mathbb R} \times \widehat{\mathbb R}$ is a finite union of separated sequences and
$D^{-}(E) \geq 1,$ where $D^{-}$ denotes the lower Beurling density, \cite{ChrDenHei1999}. (Beurling density has been
analyzed deeply in terms of Fourier frames, e.g., \cite{beur1989}, \cite{land1967}, \cite{jaff1991}, and \cite{seip1995},
and it is defined as
$$
         D^{-}(E) = {\rm lim}_{r \rightarrow \infty}\ \frac{n^{-}(r)}{r^2},
$$
where $n^{-}(r)$ is the minimal number of points from $E \subseteq {\mathbb R} \times \widehat{\mathbb R}$
in a ball of radius $r/2$.) For perspective, in the case
of $\{e_{mb}{\tau}_{na}g : m,n \in {\mathbb Z}\}$, this necessary condition is equivalent to the condition $ab \leq 1.$
It is also well-known that if $ab > 1,$ then  $\{e_{mb}{\tau}_{na}g : m,n \in {\mathbb Z}\}$ is not complete in $L^2({\mathbb R}).$
As such, it is not unexpected that $\{e_{\sigma_n}{{\tau}_{s_n}}g\}$ is incomplete if $D^{-}(E) < 1;$ however, this is not the case as has been shown by explicit construction, see \cite{BenHeiWal1995}, Theorem 2.6. Other sparse complete
Gabor systems have been constructed in \cite{rome2002} and \cite{wang2004}.

\end{example}

\begin{example}
\label{ex:exXmu}
{\it a.} Let $(X, \mathcal{A}, \mu)$ be a measure space, i.e., $X$ is a set, $\mathcal{A}$ is a $\sigma-$algebra in the power set $\mathcal{P}(X)$, and $\mu$ is a measure on $\mathcal{A}$, see \cite{BenCza2009}.  Let $H$ be a complex, separable Hilbert space.  Assume
$$
\mathcal{F} \colon X \rightarrow H
$$
is a weakly measurable function in the sense that for each $f \in H,$ the complex-valued mapping $x \mapsto \langle f, \mathcal{F}(x)\rangle$ is measurable.  $\mathcal{F}$ is a $(X, \mathcal{A}, \mu)$--{\it frame} for $H$ if
\[ \exists \ A, B > 0 \mbox{ such that } \forall \ f \in H,\quad
A \| f \|^2 \leq \int_X | \langle f, \mathcal{F}(x) \rangle |^2 \ d\mu(x) \leq B \| f \|^2. \]
Typically, $\mathcal{A}$ is the Borel algebra $\mathcal{B}(\mathbb{R}^d)$ for $X = \mathbb{R}^d$ and $\mathcal{A} = \mathcal{P}(\mathbb{Z})$ for $X = \mathbb{Z}.$  In these cases we use the terminology, $(X, \mu)$-frame.

{\it b.} Continuous and discrete wavelet and Gabor frames are special cases of $(X, \mathcal{A}, \mu)$-frames and could have been formulated as such from the time of \cite{daub1992} (1986) and \cite{HeiWal1989} (1989).
In mathematical physics the idea was introduced in \cite{kais1990},
\cite{AliAntGaz1993}, and \cite{AliAntGaz2000}.  Recent mathematical contributions are found in 
\cite{GabHan2003} and \cite{ForRau2005}.
$(X, \mathcal{A}, \mu)$-frames are sometimes referred to as {\it continuous frames.}  Also, in a slightly more concrete way we could have let $X$ be a locally compact space and $\mu$ a positive Radon measure on $X$.

{\it c.} Let $X = \mathbb{Z}, \mathcal{A} = \mathcal{P}(\mathbb{Z})$, and $\mu = c,$ where $c$ is counting measure, $c(Y) = \mbox{card}(Y)$.  Define $\mathcal{F}(n) = x_n \in H, n \in \mathbb{Z},$ for a given complex, separable Hilbert space, $H.$  We have
$$
   \forall \ f \in H, \ \int_{\mathbb{Z}} | \langle f, x_n \rangle |^2 \ d\ c(n)
 = \sum_{n \in \mathbb{Z}} \int_{\{n\}} | \langle f, x_n \rangle |^2 \ d\ c(n) = \sum_{n \in \mathbb{Z}} | \langle f, x_n \rangle |^2.
$$
Thus, {\it frames} $\{ x_n \}$ for H, as defined in Definition \ref{defn:frame}, are $(\mathbb{Z}, \mathcal{P}(\mathbb{Z}), c)$--frames.  For the present discussion we also refer to them as {\it discrete frames.}

{\it d.} Let $X = \mathbb{R}^d, \mathcal{A} = \mathcal{B}(\mathbb{R}^d)$, and $\mu = p$ a probability measure, i.e. $p(\mathbb{R}^d) = 1$; and let $H = \mathbb{R}^d.$  The measure $p$ is a {\it probabilistic frame} for $H = \mathbb{R}^d$ if
$$
 \exists \ A, B > 0 \mbox{ such that } \forall \ x \in \mathbb{R}^d \ (= H),\quad
 A \| x \|^2 \leq \int_X | \langle x, y \rangle |^2   \ d\ p(y) \leq B \| x \|^2,
$$
see \cite{ehle2012}, \cite{EhlOko2013}.
Define
$$
\mathcal{F} \colon X = \mathbb{R}^d \rightarrow H = \mathbb{R}^d
$$
by $\mathcal{F}(x) = x \in \mathbb{R}^d.$  Suppose $\mathcal{F}$ is a $(\mathbb{R}^d, \mathcal{B}(\mathbb{R}^d), p)$-frame for $H = \mathbb{R}^d.$  Then
$$
\forall \ x \in H, \quad A \| x \|^2 \leq \int_X | \langle x, y \rangle |^2   \ d \ p(y) \leq B \| x \|^2,
$$
and this is precisely the same as saying that $p$ is a probabilistic frame for $H = \mathbb{R}^d.$

Suppose we try to generalize probabilistic frames to the setting that X is locally compact, as well as being a vector space because of probabilistic applications.  This simple extension can not be effected since Hausdorff, locally compact vector spaces are, in fact, finite dimensional (F. Riesz).

{\it e.}  Let $(X, \mathcal{A}, \mu)$ be a measure space and let $H$ be a complex, separable Hilbert space.  A {\it positive operator-valued measure} ($POVM$) is a function $\pi \colon \mathcal{A} \rightarrow \mathcal{L}(H),$ where $\mathcal{L}(H)$ is the space of the bounded linear operators on $H$, such that $\pi(\emptyset) = 0, \pi(X) = I$ (Identity), $\pi(A)$ is a positive, and therefore self-adjoint (since H is a complex vector space), operator on $H$ for each $A \in \mathcal{A},$ and
$$
\forall \ \mbox{ disjoint } \{A_j \}_{j=1}^{\infty} \subseteq \mathcal{A}, \quad x, y \in H \implies \langle \pi \left(\cup_{j=1}^{\infty} A_j \right) x, y \rangle = \sum_{j=1}^{\infty} \langle \pi (A_j) x, y \rangle.
$$
$POVMs$ are a staple in quantum mechanics, see \cite{AliAntGaz2000}, \cite{BenKeb2008}
for rationale and references.  If $\{x_n\} \subseteq H$ is a 1-tight discrete frame for $H$, then it is elementary to see that the formula,
$$
 \forall \ x \in H \mbox{ and } \forall \ A \in \mathcal{P}(\mathbb{Z}), \ \pi(A) x = \sum_{n \in A} \langle x, x_n \rangle x_n,
$$
defines a $POVM.$  Conversely, if $H = \mathbb{C}^d$ and $\pi$ is a $POVM$ for $X$ countable, then by the spectral theorem there is a corresponding 1-tight discrete frame.  This relationship between tight frames and $POVMs$ extends to more general $(X, \mathcal{A}, \mu)$-frames, e.g., \cite{AliAntGaz2000}, Chapter 3.

In this setting, and related to {\it probability of quantum detection error,} $P_e$,  which is defined in terms of $POVMs,$ Kebo and one of the authors have proved the following for $H = \mathbb{C}^d, \{ y_j \}_{j=1}^{N} \subseteq H,$ and $\{\rho_j > 0 \}_{j=1}^N, \sum_{j=1}^N \rho_j = 1 \colon$ there is a 1-tight discrete frame $\{x_n\}_{n=1}^N \subseteq H$ for $H$ that minimizes $P_e$, \cite{BenKeb2008}, Theorem A.2.

{\it f.} Let $X = \mathbb{R}^{2d}$ and let $H = L^2(\mathbb{R}^d)$.  Given $g \in L^2(\mathbb{R}^d)$ and define the function
\begin{align*}
 \mathcal{F} \colon  \mathbb{R}^{2d} & \rightarrow L^2( \mathbb{R}^d) \\
  (x, \omega) & \mapsto e^{2 \pi i t \cdot \omega} \ g(t - x).
\end{align*}
$\mathcal{F}$ is a $(\mathbb{R}^{d}, \mathcal{B}(\mathbb{R}^{2d}), m)$-frame for $L^2(\mathbb{R}^{2d}),$ where $m$ is Lebesgue measure on $\mathbb{R}^{2d}$; and, in fact, it is a tight frame for $L^2(\mathbb{R}^{d})$ with frame constant $A = B = \| g \|_2^2.$  To see this we need only note the following consequence of the orthogonality relations for the $STFT$:
\begin{equation}
\label{eq:STFT_orthogonality}
\| V_gf \|_2 = \| g \|_{L^2(\mathbb{R}^{d}) } \| f \|_{ L^2(\mathbb{R}^{d})}.
\end{equation}
Equation (\ref{eq:STFT_orthogonality}) is also used in the proof of (\ref{eq:InversionSTFT}).

{\it g.} Clearly, Theorems \ref{thm:1NonUniform}, \ref{thm:groch}, and \ref{thm:stft-frame} can be formulated in terms of
$(X,\mu)$--frames.
\end{example}


\section{Examples and modifications of Beurling's method}
\label{sec:ex}

\subsection{Generalizations of Beurling's Fourier frame theorem}
\label{sec:123-weights}

Using more than one measure, we can extend Theorem \ref{theorem:balayage2} to more general types of Fourier frames.  For clarity we give the result for three simple measures.

\begin{lem}\label{lemma:Fourier1}
Given the notation and hypotheses of Theorems \ref{thm:balayage3} and \ref{thm:balayage4}.  Then,
\[   \forall f \in PW_{\Lambda}\setminus \{0\},\, \widehat{f} = F,\]
\[ \sum_{x \in E} \left| \int a_{x}(y) h(x-y) f(y) \,dy \right|^2 \leq [K(E,\Lambda_{\epsilon}) \norm{h}_2]^2 \int_{\Lambda} |F(\gamma)|^2 \ d\gamma.\]

\begin{proof}
We compute:

\begin{align*}
& \sum_{x \in E} \left|\int a_x(y)h(x-y)f(y)\,dy \right|^2 \\
&\leq \sum_{x\in E} \left| \left(\int | a_x(y)^{1/2} h(x-y)|^2 \,dy\right)^{1/2} \left(\int |a_x(y)^{1/2} f(y)|^2 \,dy \right)^{1/2} \right|^{2}\\
&\leq \sup_{x\in E} \left( \int | a_x(y) | | h(x-y)|^2 \,dy\right) \left(\sum_{x\in E} \int |a_x(y)| | f(y)|^2 \,dy \right)\\
&\leq K(E, \Lambda_\epsilon) \sup_{x\in E} \left(\int |a_{x}(y)| |h(x-y)|^2 \,dy \right) \int_{\Lambda} |F(\gamma)|^2\,d\gamma \\
&\leq K(E, \Lambda_\epsilon)^{2} \norm{h}_2^2\int_{\Lambda} |F(\gamma)|^2\,d\gamma,
\end{align*}
where we have used the Plancherel theorem to obtain the third inequality.

\end{proof}

\end{lem}

\begin{thm}
\label{theorem:Generalized_fourier_frames}
Let $E =\{x_n\} \subseteq \RR^d$ be a separated sequence, and let $\Lambda \subseteq \Rhat^d$ be a compact S-set of strict multiplicity.
Assume that $\Lambda$ is a compact, convex set, that is symmetric about $0 \in \Rhat^d$. If balayage is possible for $(E, \Lambda)$, then
\[\exists \ A, B > 0 \text{ \rm{ such that} } \forall f \in PW_{\Lambda}\setminus \{0\},\, F = \widehat{f},\]
\begin{align}\begin{split}\label{eq:a}
& A^{1/2} \frac{\int_{\Lambda} |F(\gamma) + F(2 \gamma) + F(3 \gamma)|^2 \ d\gamma}{\left( \int_{\Lambda} |F(\gamma)|^{2} \ d \gamma \right)^{1/2}}\\
&\leq \left(\sum_{x \in E} |f(x)|^{2}\right)^{1/2} + \frac{1}{2}\left(\sum_{x \in E} |f(\frac{1}{2} x)|^{2}\right)^{1/2} + \frac{1}{3}\left(\sum_{x \in E} |f(\frac{1}{3} x)|^{2}\right)^{1/2} \\
&\leq \ B^{1/2} \left(\int_{\Lambda} |F(\gamma)|^2 \,d\gamma \right)^{1/2}.
\end{split}
\end{align}
\vspace{0.2in}
\begin{proof}
By hypothesis, we can invoke Theorem \ref{theorem:balayage1} to choose $\epsilon > 0$ so that
balayage is possible for $(E, \Lambda_{\epsilon})$, i.e., $K(E, \Lambda_{\epsilon}) < \infty$.
For this $\epsilon > 0$ and appropriate $\Omega,$ we use Theorem \ref{thm:balayage3} to choose $h \in
L^{1}(\mathbb{R}^{d})$ for which $h(0) = 1, \supp(\widehat{h}) \subseteq \overline{B(0,\epsilon)},$ and $|h(x)| = O(e^{- \Omega(\|x\|)} ), \|x\| \rightarrow \infty.$

Therefore, for a fixed $y \in \mathbb{R}^{d}$ and $g \in \mathcal{C}(\Lambda)$,
Theorem \ref{thm:balayage4} allows us to assert that
\begin{eqnarray*}
& & g(y) + g(2y) + g(3y) \\
& = & \sum_{x \in E} g(x) \left( a_x(y) h(x-y) + a_x(2y) h(x - 2y) + a_x(3y) h(x - 3y) \right)
\end{eqnarray*}
and
\[ \sum_{x \in E} \left|a_x(j y)\right| \leq K(E, \Lambda_{\epsilon}), \ j = 1,2,3. \]
Hence, if $\gamma \in \Lambda$ is fixed and $g(w) = e^{-2 \pi i w \cdot \gamma},$ then
\begin{eqnarray*}
& & e^{-2 \pi i y \cdot \gamma} + e^{-2 \pi i (2y) \cdot \gamma} + e^{-2 \pi i (3y) \cdot \gamma} \\
& = & \sum_{x \in E} \left( a_x(y)  h(x - y) + a_x(2y)  h(x - 2y) + a_x(3y)  h(x - 3y) \right) \ e^{-2 \pi i x \cdot \gamma},\\
\end{eqnarray*}
which we write as
\[ \sum_{x \in E} b_x(y) e^{-2 \pi i x \cdot \gamma}.\]

Since $L^{1}(\mathbb{R}^d) \cap PW_{\Lambda}$ is dense in $PW_{\Lambda},$ we take
$f \in L^{1}(\mathbb{R}^d) \cap PW_{\Lambda}$ in the following argument 
without loss of generality.  We compute
\begin{eqnarray*}
& & \sum_{x \in E} e^{-2 \pi i x \cdot \gamma} \int b_x(y) f(y) \,dy  \\
& = & \int f(y) \left( \sum_{x \in E} b_x(y)  e^{-2 \pi i x \cdot \gamma} \right)  \,dy \\
 & = & \int f(y) \left( e^{-2 \pi i y \cdot \gamma} + e^{-2 \pi i (2y) \cdot \gamma} + e^{-2 \pi i (3y) \cdot \gamma} \right) \,dy \\
& = & F(\gamma) + F(2 \gamma) + F(3 \gamma) = J_F(\gamma).
\end{eqnarray*}
As such, we have
\[J_F(\gamma) = \sum_{x \in E} \widetilde{f}(x)  e^{-2 \pi i x \cdot \gamma}, \quad \text{where }\widetilde{f}(x) = \int b_x(y) f(y) \,dy.\]
Next, we compute the following inequality for the inner product
$ \langle J_F, J_F \rangle_{\Lambda}:$
\begin{align}\label{eq:b}
&\int_{\Lambda} J_F(\gamma) \overline{J_F(\gamma) } \ d\gamma = \ \int_{\Lambda} J_F(\gamma) \left( \sum_{x \in E} \overline{\widetilde{f}(x)} e^{2 \pi i x \cdot \gamma} \right) d\gamma \nonumber\\
&=\sum_{x \in E} \overline{\widetilde{f}(x)} \left( \int_{\Lambda} J_F(\gamma) e^{2 \pi i x \cdot \gamma} \,d\gamma \right) \ =  \ \sum_{x \in E} \overline{\widetilde{f}(x)} \left( f(x) + \frac{1}{2} f(\frac{x}{2}) + \frac{1}{3}f(\frac{x}{3}) \right)\\
&\leq \left( \sum_{x \in E} | \widetilde{f}(x)|^2 \right)^{1/2} \left( \sum_{x \in E} \left| f(x) + \frac{1}{2} f(\frac{x}{2}) + \frac{1}{3}f(\frac{x}{3}) \right|^2 \right)^{1/2}  \nonumber\\
&\leq \left(\sum_{x \in E} |\widetilde{f}(x)|^2 \right)^{1/2}  \left[ \left(\sum_{x \in E} |f(x)|^2 \right)^{1/2}  + \frac{1}{2} \left(\sum_{x \in E} |f(\frac{x}{2})|^2 \right)^{1/2} + \frac{1}{3} \left(\sum_{x \in E} |f(\frac{x}{3} )|^2 \right)^{1/2}   \right]\nonumber
\end{align}
by H{\"o}lder's and Minkowski's inequalities.
 Further, there is $A > 0$ such that
\begin{equation}
\label{eq:c}
\sum_{x \in E} |\widetilde{f}(x)|^2 \leq \frac{1}{A} \int_{\Lambda} |F(\gamma)|^2 \,d\gamma. \hspace{2in}
\end{equation}
This is a consequence of Lemma \ref{lemma:Fourier1}.  Combining the definition of $J_F$ with the inequalities
\eqref{eq:b} and \eqref{eq:c} yield the first inequality of \eqref{eq:a}.

The second inequality of \eqref{eq:a} only requires the assumption that $E$ be separated, and, as such, it is a consequence of the
Plancherel-P\'{o}lya theorem, which asserts that if $E$ is separated, then
\[\exists \ B_j \text{ such that } \forall \ f \in PW_{\Lambda}, \]
\[\sum_{x \in E} \left| f\left(\frac{x}{j}\right)\right|^2 \leq B_j \ \| f \|_{2}^{2}, \ j = 1,2,3,\]
see \cite{bene1992}, pages 474-475, \cite{land1967}, \cite{SteWei1971}, pages 109-113.
\end{proof}
\end{thm}

Theorem \ref{theorem:Generalized_fourier_frames} can be generalized extensively.\\

\begin{example}
Given the setting of Theorem \ref{theorem:Generalized_fourier_frames}.

{\it a.} Define the set $\{e_{j,x}^\vee : j = 1, 2, 3 \text{ and } x \in E\}$ of functions on $\RR^d$ by
$$e_{j,x}(\gamma) = \frac{1}{j} \mbbu_\Lambda (\gamma)e^{-2 \pi i (1/j)x \cdot \gamma},$$
and define the mapping $S: PW_\Lambda \rightarrow PW_\Lambda$ by
$$Sf = \sum_{j=1}^3 \sum_{x \in E} \inner{f}{e_{j,x}^\vee}e_{j,x}^\vee.$$
We compute
$$\forall f \in PW_\Lambda, \quad \inner{Sf}{f} = \sum_{j=1}^3 \frac{1}{j^2} \sum_{x \in E} \left|f\left(\frac{x}{j}\right)\right|^2.$$

{\it b.} Let $f \in PW_\Lambda$, $\widehat{f} = F$, and define $J_F(\gamma) = F(\gamma) + F(2 \gamma) + F(3 \gamma)$. Since $(a + b + c)^2 \leq 3(a^2+b^2+c^2)$ for $a, b, c \in \RR$, Theorem \ref{theorem:Generalized_fourier_frames} and part {\it a} allow us to write the frame-type inequality,
\begin{equation}
\label{eq:x}
\frac{A}{3} \frac{\inner{J_F}{J_F}^2}{\norm{F}_2}\leq \inner{Sf}{f} = \norm{Lf}_{\ell^2}^2 \leq B \norm{f}_2^2,
\end{equation}
where $Lf = \{\inner{f}{e_{j,x}^\vee}: j = 1, 2, 3 \text{ and } x \in E\}$ so that $S = L^\ast L$. The inequalities \eqref{eq:x} do not a priori define a frame for $PW_\Lambda$. However, $\{e_{j,x}: j = 1, 2, 3 \text{ and } x \in E\}$ is a frame for $PW_\Lambda$ with frame operator $S$. This is a consequence of Theorem \ref{theorem:balayage2}.
\end{example}

\begin{thm}
\label{theorem:weightedbalayage}
Let $E = \{x_{n}\} \subseteq \mathbb{R}^d$ be a separated sequence, and let
$\Lambda \subseteq \widehat{\mathbb{R}}^d$ be an S-set of strict multiplicity.
Assume that $\Lambda$ is a compact, convex set, that is symmetric about $0 \in \Rhat^d$.
Further, let $G \in L^\infty(\RR^d)$ be non-negative on $\Rhat^d$. If balayage is
possible for $(E, \Lambda)$, then
\[ \exists \ A, B > 0,  \text{ such that } \forall \ f \in PW_\Lambda \setminus \{0\}, F = \widehat{f},\]
\begin{align}\label{eq:X}
A \frac{\left(\int_{\Lambda} |F(\gamma)|^2 \ G(\gamma) \,d\gamma\right)^2}{\int_{\Lambda} |F(\gamma)|^2 \,d\gamma}
& \leq \sum_{x \in E} |\left(F \ G \right)^\vee (x) |^{2}\\
& \leq B \int_\Lambda \left| F(\gamma) \right|^2 \,d\gamma \nonumber.
\end{align}
We can take $A = 1/\left(K(E, \Lambda_\epsilon)\norm{h}_2^2 \right)$ and $B = B_1
\norm{G}_\infty^2$, where $B_1$ is the Bessel bound in the Plancherel-P\'{o}lya theorem for $PW_\Lambda$.

\begin{proof}
By hypothesis, we can invoke Theorem \ref{theorem:balayage1} to choose $\epsilon > 0$
so that balayage is possible for $(E, \Lambda_\epsilon)$, i.e., $K(E, \Lambda_\epsilon) < \infty$.
For this $\epsilon >0$ and appropriate $\Omega$, we use Theorem \ref{thm:balayage3} to choose $h \in L^1(\RR^d)$ for which $h(0) = 1$, $\supp{\widehat{h}} \subseteq \overline{B(0, \epsilon)}$, and $|h(x)| = O(e^{-\Omega(\norm{x})}), \norm{x} \rightarrow \infty$.
Consequently, we have
\[\forall \ y \in \mathbb{R}^d \text{ and } \forall \ \gamma \in \Lambda,\]
\[e^{-2 \pi i y \cdot \gamma} = \sum_{x \in E} a_x(y)  h(x - y) e^{-2 \pi i x \cdot \gamma}, \text{ where } \sum_{x \in E} |a_x(y)| \leq K(E, \Lambda_{\epsilon}).\]
If $f \in PW_\Lambda$, $\widehat{f} = F$, and noting that $F \in L^1({\Rhat^d})$, we have the following computation:
\begin{eqnarray}\label{eq:y}
& & \int_{\Lambda} |F(\gamma)|^2 G(\gamma) \ d \gamma  \nonumber \\
& = & \int_{\Lambda} F(\gamma)  G(\gamma) \left(\int \overline{f(w)} \left( \sum_{x \in E} a_x(w) h(x - w)  e^{2 \pi i x \cdot \gamma} \right) \,dw \right) \,d\gamma \nonumber\\
& = & \sum_{x \in E} \left(\int_{\Lambda} F(\gamma) G(\gamma)  e^{2 \pi i x \cdot \gamma} \ d\gamma \right) \left(\int \overline{f(w)} a_{x}(w)  h(x - w) \ dw \right) \\
& \leq & \left(\sum_{x \in E} | (F G) ^\vee (x)|^2  \right)^{1/2} \left( \sum_{x \in E} \left| \int \overline{f(w)} a_{x}(w)  h(x - w) \ dw \right|^2 \right)^{1/2}\nonumber\\
& \leq & K(E, \Lambda_{\epsilon}) \norm{h}_{2} \left( \int_{\Lambda} | F(\gamma) |^2 \, d\gamma \right)^{1/2} \left(\sum_{x \in E} \left| (F G)^\vee (x) \right|^2 \right)^{1/2},\nonumber
\end{eqnarray}
where the last step is a consequence of Lemma \ref{lemma:Fourier1}. Clearly, \eqref{eq:y} gives the first inequality of \eqref{eq:X}. As in Theorem \ref{theorem:Generalized_fourier_frames}, the second inequality of \eqref{eq:X} only requires the assumption that $E$ be separated, and, as such, it is a consequence of the Plancherel-P\'{o}lya theorem for $PW_\Lambda$.
\end{proof}
\end{thm}

Theorem \ref{theorem:weightedbalayage} is an elementary generalization of the classical result for the case $G = 1$ on $\RR$, and itself has significant generalizations to other weights $G$. We have not written $(FG)^\vee$ as a convolution since for such generalizations there are inherent subtleties in defining the convolution of distributions, e.g., \cite{schw1966}, Chapitre VI, \cite{meye1981}, see \cite{bene1997}, pages 99-102, for contributions of Hirata and Ogata, Colombeau, et al. Even in the case of Theorem \ref{theorem:weightedbalayage}, $G^\vee = g$ is in the class of pseudo-measures, which themselves play a basic role in spectral synthesis \cite{bene1975}.

\subsection{A bounded operator $B: L^{p}(\mathbb{R}^{d}) \rightarrow l^{p}(E), \ p > 1$} 
\label{sec:landau}

{\it a.}
In Example \ref{ex:fouierframebalayage}{\it b} we proved the lower frame bound assertion of Theorem \ref{theorem:balayage2}. This can also be achieved using Beurling's generalization of balayage to so-called linear balayage operators $B$, see \cite{beur1989}, pages 348-350. 

In fact, with this notion and assuming the hypotheses of Theorem \ref{thm:balayage4}, Beurling proved that the mapping,
\begin{eqnarray*}
   L^{p}(\mathbb{R}^d) & \longrightarrow & l^{p}(E),\quad p > 1,\\
   k & \mapsto & \{k_{x} \}_{x \in E},
\end{eqnarray*}
where
$$
\forall \ x \in E, \quad k_x = \int_{\mathbb{R}^d} a_{x}(y) h(x-y) k(y) \ dy, 
$$
has the property that
$$
 \exists \ C_p > 0 \text{ such that } \forall \ k \in L^{p}(\mathbb{R}^d),
$$
\begin{equation}
\label{eqn:Landau1}
  \sum_{x \in E} |k_x |^{p} \leq C_p \int |k(y)|^{p} dy.
\end{equation}

Let $p = 2$ and fix $f \in PW_{\Lambda}.$  We shall use (\ref{eqn:Landau1}) and the definition of norm to obtain the desired lower frame bound.
This is H.J. Landau's idea.  Set
$$
   I_k = \int_{\Lambda} F(\gamma) \overline{K(\gamma)} d\gamma, \quad \quad \widehat{f} = F, 
$$
where $K^{\vee} = k \in L^{2}({\mathbb R}^d).$   By balayage, we have
\[K(\gamma) = \sum_{x \in E} k_x e^{-2 \pi i x \cdot \gamma} \text{ on } \Lambda; \]
and so,
\[I_k = \sum_{x \in E} f(x) \overline{k_x}, \]
allowing us to use (\ref{eqn:Landau1}) to make the estimate,
\[ |I_k|^2 \leq C \|K \|_{2}^{2} \ \sum_{x \in E}|f(x)|^2.\]
By definition of $\|f \|_{2}$, we have
\[ \| f \|_{2} = \sup_{K} \frac{| I_{K} |}{\| K \|_{2}} \leq C \left( \sum_{x \in E}|f(x)|^2 \right)^{1/2},\]
and this is the lower frame bound inequality with bound $A = 1/C^2.$\\

Because of this approach we can think of balayage as "$l^{2}-L^{2}$ balayage".

{\it b.}
Motivated by part {\it a}, we shall say that $l^{1}-L^{2}$ {\it balayage is possible} for $(E, \Lambda)$, where $E$ is separated and $\Lambda$ is a compact set of positive measure $| \Lambda |$, if
\[ \exists \ C > 0 \text{ such that } \forall \ k \in L^{2}(\mathbb{R}^d), \widehat{k} = K,\]
\[\sum_{x \in E} | k_x | \leq C \int_{\Lambda} | K(\gamma)|^2  d \gamma\]
and
\[K(\gamma) = \sum_{x \in E} k_x e^{-2 \pi i x \cdot \gamma} \text{ on } \Lambda. \]
For fixed $f \in PW_{\Lambda}$ and using the notation of part {\it a,} we have
\begin{equation}\label{eqn17}
|I_k|^2 \leq \sum_{x \in E} |k_x|^2 \sum_{x \in E} | f(x) |^2.
\end{equation}
An elementary calculation gives
\[\sum_{x \in E} |k_x|^2 \leq C^2 | \Lambda | \int_{\Lambda} | K(\gamma)|^2 d\gamma, \]
which, when substituted into (\ref{eqn17}), gives
\[\frac{1}{C^2 | \Lambda |}  \left( \frac{| I_K |^2}{ \int_{\Lambda} | K(\gamma)|^2 d\gamma} \right) \leq  \sum_{x \in E} | f(x) |^2.\]
We obtain the desired lower frame inequality with bound $ A = 1/(C^2 | \Lambda |). $


\section{Pseudo-differential operator frame inequalities}
\label{sec:pdo}

Let $\sigma \in \mathcal{S}^{\prime}(\mathbb{R}^d \times \widehat{\mathbb{R}}^d).$  The operator, $K_{\sigma},$ formally defined as
\[ (K_{\sigma} f)(x) = \int \sigma(x, \gamma) \widehat{f}(\gamma) e^{2 \pi i x \cdot \gamma} \ d\gamma, \]
is the \emph{pseudo-differential operator} with Kohn-Nirenberg symbol, $\sigma$, see \cite{groc2001} Chapter 14, \cite{groc2006} Chapter 8, \cite{horm1979}, and \cite{stei1993}, Chapter VI. 
For consistency with the notation of the previous sections, we shall define pseudo-differential operators, $K_s,$ with tempered distributional 
Kohn-Nirenberg symbols, $s \in  \mathcal{S}^{\prime}(\mathbb{R}^d \times \widehat{\mathbb{R}}^d),$ as
\[ (K_{s} \widehat{f})(\gamma) = \int s(y, \gamma) f(y) e^{-2 \pi i y \cdot \gamma} \ dy. \]

Further, we shall actually deal with Hilbert-Schmidt operators, $K \colon L^2(\widehat{\mathbb{R}}^d) \rightarrow L^2(\widehat{\mathbb{R}}^d)$;
and these, in turn, can be represented as $K = K_s,$ where $s \in L^2(\mathbb{R}^d \times \widehat{\mathbb{R}}^d)$.
Recall that $K \colon L^2(\widehat{\mathbb{R}}^d) \rightarrow L^2(\widehat{\mathbb{R}}^d)$ is a \emph{Hilbert-Schmidt operator} if
\[ {\sum}_{n=1}^{\infty} \| K e_n \|_2^2 < \infty \]
for some orthonormal basis, $\{e_n\}_{n=1}^{\infty},$ for $L^2(\widehat{\mathbb{R}}^d)$, in which case the \emph{Hilbert-Schmidt norm} of $K$ is defined as
\[ \| K \|_{HS} = \left(\sum_{n=1}^{\infty} \| K e_n \|_2^2  \right)^{1/2}, \]
and $\| K \|_{HS}$ is independent of the choice of orthonormal basis.  The first theorem about Hilbert-Schmidt operators is the following \cite{RieNag1955}:

\begin{thm}
If $K \colon L^2(\widehat{\mathbb{R}}^d) \rightarrow L^2(\widehat{\mathbb{R}}^d)$ is a bounded linear mapping and 
$(K \widehat{f})(\gamma) = \int m(\gamma, \lambda) \widehat{f}(\lambda) \ d\lambda,$ for some measurable function $m$,
then $K$ is a Hilbert-Schmidt operator if and only if $m \in L^2(\widehat{\mathbb{R}}^{2d})$ and, in this case, $\| K \|_{HS} = \| m \|_{L^2(\mathbb{R}^{2d})}.$
\end{thm}

The following is our result about pseudo-differential operator frame inequalities.

\begin{thm}\label{thm:pseudoD}
Let $E = \{x_n\} \subseteq \mathbb{R}^d$ be a separated sequence, that is symmetric about $0 \in \mathbb{R}^d$; and
let $\Lambda \subseteq \widehat{\mathbb{R}}^d$ be an S-set of strict multiplicity, that is compact, convex, 
and symmetric about $0 \in \widehat{\mathbb{R}}^d$.  Assume balayage is possible for $(E, \Lambda)$.  Further, let $K$ be a Hilbert-Schmidt
operator on $L^2(\widehat{\mathbb{R}}^d)$ with pseudo-differential operator representation,
\[
(K \widehat{f})(\gamma) = (K_{s} \widehat{f})(\gamma) = \int s(y, \gamma) f(y) e^{-2 \pi i y \cdot \gamma} \ dy,
\]
where $s_{\gamma}(y) = s(y, \gamma) \in L^2(\mathbb{R}^d \times \widehat{\mathbb{R}}^d)$ is the Kohn-Nirenberg symbol and where we make the
further assumption that
\begin{equation}\label{eqn:assumpThm5.2}
 \forall \gamma \in \widehat{\mathbb{R}}^d, \quad s_{\gamma} \in C_b(\mathbb{R}^d)
 \quad and \quad {\rm supp}\,(s_{\gamma} e_{- \gamma})^{\widehat{}} \subseteq \Lambda. 
 \end{equation}
Then,
\[ \exists A,\,B > 0 \quad \text{such that} \quad \forall f \in L^2(\mathbb{R}^d) \backslash \{0\}, \]
\begin{equation}\label{eqn:Thm5.2}
 A \frac{ \| K_s \widehat{f} \|_2^4}{\| f \|_2^2} \leq \sum_{x \in E} | \langle (K_s \widehat{f})(\cdot), \overline{s(x, \cdot)} \ e_x(\cdot) \rangle |^2  \leq B \ \| s \|_{L^2(\mathbb{R}^d \times \widehat{\mathbb{R}}^d)}^{2} \| K_s \widehat{f} \|_2^2.   
 \end{equation}
\end{thm}

\begin{proof}
{\it a.} In order to prove the assertion for the lower frame bound, we
first combine the pseudo-differential operator $K_s$, with Kohn-Nirenberg symbol $s$, and balayage to compute
\begin{align}\label{eqn:5A}
& \int | (K_s \widehat{f})(\gamma) |^2 \ d\gamma = \int \overline{(K_s \widehat{f})(\gamma)} (K_s \widehat{f})(\gamma) \ d\gamma \\
= & \int \overline{(K_s \widehat{f})(\gamma)} \left( \int s(y, \gamma) f(y) e^{-2 \pi i y \cdot \gamma} \ dy \right) d\gamma \nonumber \\
= & \int \overline{(K_s \widehat{f})(\gamma)} \left( \int f(y) k(y, \gamma) \ dy \right) d\gamma \nonumber \\
= & \int \overline{(K_s \widehat{f})(\gamma)} \left( \int f(y) \left( \sum_{x \in E} k(x, \gamma) a_x(y, \gamma) h(x-y) \right) dy \right) d\gamma, \nonumber
\end{align}
where $k_{\gamma}(y) = k(y, \gamma) = s(y, \gamma) e^{-2 \pi i y \cdot \gamma}$ on $\mathbb{R}^d$ and $k_{\gamma} \in \mathcal{C}(\Lambda)$
for each fixed $\gamma \in \widehat{\mathbb{R}}^d$, and where
\begin{equation}
\label{eqn:5B}
            \sup_{\gamma \in \widehat{\mathbb{R}}^d} \sup_{y \in {\mathbb{R}}^d} 
            \sum_{x \in E}\,| a_x(y, \gamma)|         \leq K(E, \Lambda_{\epsilon}) = C < \infty.
\end{equation}
Because of Theorems \ref{thm:balayage3} and \ref{thm:balayage4}, we do not need to have the function $h$ depend on $\gamma \in \widehat{\mathbb{R}}^d$.  
Further, because of (\ref{eqn:5B}) and estimates we shall make, we can write $a_x(y, \gamma) = a_x(y)$.

Thus, the right side of (\ref{eqn:5A}) is
\begin{align}
\label{eqn:5C}
& \int f(y) \left[ \sum_{x \in E} a_x(y) h(x-y) \left( \int \overline{(K_s \widehat{f})(\gamma)} k(x, \gamma) \ d\gamma \right) \right] dy \\
= & \sum_{x \in E} \left( \int f(y) a_x(y) h(x-y) \ dy \int   \overline{(K_s \widehat{f})(\gamma)} k(x, \gamma) \ d\gamma \right) \nonumber \\
\leq & \left(  \sum_{x \in E} \left| \int f(y) a_x(y) h(x-y) \ dy \right|^2 \right)^{1/2} \left(  \sum_{x \in E} \left|   \overline{(K_s \widehat{f})(\gamma)} k(x, \gamma)   \right|^2 \right)^{1/2} . \nonumber
\end{align}
Note that, by H\"{o}lder's inequality applied to the integral, we have
\begin{align}
\label{eqn:5D}
& \sum_{x \in E} \left| \int f(y) a_x(y) h(x-y) \ dy \right|^2 \\
\leq & \sum_{x \in E} \left| \left(  \int | a_x(y) | | h(x-y) |^2 \ dy \right)^{1/2} \left( \int | f(y) |^2 | a_x(y)| \ dy \right)^{1/2} \right|^2 \nonumber \\
\leq & \sum_{x \in E} \left(  C \int | h(x-y) |^2 \ dy \right) \left( \int | f(y) |^2 | a_x(y)| \ dy \right)  \nonumber \\
\leq & C \| h \|_2^2 \int \left( ( \sum_{x \in E} | a_x(y) | )  \left| f(y) \right|^2   \ dy \right)  \nonumber \\
\leq & C^2 \| h \|_2^2 \| f \|_2^2. \nonumber
\end{align}

Combining (\ref{eqn:5A}), (\ref{eqn:5C}), and (\ref{eqn:5D}), we obtain
\[ 
\| K_s \widehat{f} \|_2^2 \leq C \|h\|_2 \| f \|_2 \left( \sum_{x \in E} \left|  \int (K_s \widehat{f})(\gamma) k(x, \gamma) \ d\gamma \right|^2  \right)^{1/2}.
\]
Consequently, setting $A = 1/(C\|h\|_2)^2$, we have
\begin{align*}
\forall f \in L^2(\mathbb{R}^d) \backslash \{0\}, \quad A \frac{\| K_s \widehat{f} \|_2^4}{\| f \|_2^2} & \leq \sum_{x \in E} \left| \int (K_s \widehat{f} )(\gamma) s(x, \gamma) e^{-2 \pi i x \cdot \gamma} \ d\gamma \right|^2 \\
& = \sum_{x \in E}  | \langle (K_s \widehat{f})(\cdot), \overline{s(x, \cdot)} e_x(\cdot) \rangle |^2
\end{align*}
and the assertion for the lower frame bound is proved.

{\it b.i.} In order to prove the assertion for the upper frame bound, we begin by formally defining
\[
   \forall f \in L^2({\mathbb R}^d), \quad (I_s \widehat{f})(x) = \int s(x,\gamma) (K_s \widehat{f})(\gamma)  e^{-2 \pi i x \cdot \gamma} \ d\gamma,
\]
which is the inner product in (\ref{eqn:Thm5.2}).

Note that $I_s \widehat{f} \in L^2(\mathbb{R}^d).$  In fact, we know $K_s \widehat{f} \in L^2(\widehat{\mathbb{R}}^d)$ and $s \in  L^2(\mathbb{R}^d \times \widehat{\mathbb{R}}^d)$ so that 
\[ | I_s \widehat{f}(x) |^2 \leq \int | s(x, \gamma) |^2 \ d\gamma \int | K_s \widehat{f}(\gamma) |^2 \ d\gamma \]
by H\"{o}lder's inequality, and, hence, 
\begin{equation}\label{eqn:I_s_f_hat}
\| I_s \widehat{f} \|_2^2 \leq \| s \|_{ L^2(\mathbb{R}^d \times \widehat{\mathbb{R}}^d)}^2 \| K_s \widehat{f}  \|_2^2.
\end{equation}
{\it b.ii.} We shall now show that supp$((I_s \widehat{f})\ \widehat{} \ ) \subseteq \Lambda$, and to this end we use (\ref{eqn:assumpThm5.2}).
We begin by computing
 \begin{align*}
 (I_s\widehat{f})\ \widehat{}\ (\omega) & = \int \left( \int s(y, \gamma) (K_s \widehat{f})(\gamma) \ e^{-2 \pi i y \cdot \gamma} \ d\gamma \right) e^{-2 \pi i y \cdot \omega} \ dy \\
 & = \int  (K_s \widehat{f})(\gamma) \left( \int k_{\gamma}(y) e^{-2 \pi i y \cdot \omega} \ dy \right) \ d\gamma \\
 & = \int (K_s \widehat{f})(\gamma) (k_{\gamma})^{\widehat{}} (\omega) \ d\gamma,
 \end{align*}
where \[k_{\gamma}(y) = k(y, \gamma) =   s(y,\gamma)  e^{-2 \pi i y \cdot \gamma} = (s_{\gamma} e_{- \gamma})(y), \]
as in part {\it a}.  Also, supp$(k_{\gamma})^{\widehat{}} \ \subseteq \Lambda$ by our assumption, (\ref{eqn:assumpThm5.2}); that is, for each $\gamma \in \widehat{\mathbb{R}}^d, (k_{\gamma})^{\widehat{}} \ = 0$ a.e. on $\widehat{\mathbb{R}}^d \backslash \Lambda$.

Since supp$(I_s\widehat{f})\ {\widehat{}} \ $ is the smallest closed set outside of which $(I_s\widehat{f})\ \widehat{}\ $ is 0 a.e., we need only show that if supp$(L) \subseteq \widehat{\mathbb{R}}^d \backslash \Lambda$ then
\[ \int L(\omega) (I_s\widehat{f})\ \widehat{}\ (\omega) \ d\omega = 0. \]
This follows because
\[ \int L(\omega) (I_s\widehat{f})\ \widehat{} \ (\omega) \ d\omega = \int (K_s \widehat{f})(\gamma) \left( \int L(\omega) (k_{\gamma})^{\widehat{}} (\omega) \ d\omega \right) \ d\gamma\]
and $(k_{\gamma})^{\widehat{}} = 0$ on $\widehat{\mathbb{R}}^d \backslash \Lambda.$

{\it b.iii.} Because of parts ${\it b. i}$ and ${\it b. ii}$, we can invoke the P{\'o}lya-Plancherel theorem to assert the existence of $B > 0$ such that
\[ \forall f \in L^2(\mathbb{R}^d), \quad \sum_{x \in E} | (I_s \widehat{f}) (x) | \leq B \| I_s \widehat{f} \|_2^2, \]
and the upper frame inequality of (\ref{eqn:Thm5.2}) follows from (\ref{eqn:I_s_f_hat}).


\end{proof}

\begin{example}
We shall define a Kohn-Nirenberg symbol class whose elements $s$ satisfy the hypotheses of Theorem \ref{thm:pseudoD}.

Choose $\{ \lambda_j \} \subseteq \text{int}(\Lambda), a_j \in C_b(\mathbb{R}^d) \cap L^2(\mathbb{R}^d),$ 
and $b_j \in C_b(\widehat{\mathbb{R}}^d) \cap L^2(\widehat{\mathbb{R}}^d)$ with the following properties:\\

{\it i.} \; $ \sum_{j=1}^{\infty} | a_j(y) b_j(\gamma)|$ is uniformly bounded and converges uniformly on $\mathbb{R}^d \times \widehat{\mathbb{R}}^d$;\\

{\it ii.} \; $ \sum_{j=1}^{\infty} \| a_j \|_2 \| b_j \|_2 < \infty;$\\

{\it iii.} \; $ \forall j = 1, \ldots, \ \exists \epsilon_j > 0$ such that $\overline{B(\lambda_j, \epsilon_j)} \subseteq \Lambda$ and supp$(\widehat{a}_j) \subseteq \overline{B(0, \epsilon_j)}.$\\

\noindent
These conditions are satisfied for a large class of functions $a_j$ and $b_j$.

  The Kohn-Nirenberg symbol class consisting of functions, $s$, defined as
\[ 
 s(y, \gamma) = \sum_{j=1}^{\infty} a_j(y) b_j(\gamma) e^{-2 \pi i y \cdot \lambda_j} 
 \]
 satisfy the hypotheses of Theorem \ref{thm:pseudoD}.  To see this, first note that condition \emph{i} tells us that, if we set $s_{\gamma}(y) = s(y, \gamma),$
 then
 \[ \forall \gamma \in \widehat{\mathbb{R}}^d, \quad s_{\gamma} \in C_b(\mathbb{R}^d). \]
 Condition \emph{ii} allows us to assert that $s \in L^2(\mathbb{R}^d \times \widehat{\mathbb{R}}^d)$ since we can use Minkowski's inequality
 to make the estimate,
 \[ \| s \|_{L^2(\mathbb{R}^d \times \widehat{\mathbb{R}}^d)} \leq \sum_{j=1}^{\infty} \left( \int \int \left| b_j(\gamma) a_j(y) e^{-2 \pi i y \cdot (\lambda_j - \gamma)} \right|^2 \ dy \ d\gamma \right)^{1/2} = \sum_{j=1}^{\infty} \| a_j \|_2 \|b_j \|_2 .
 \]
 Finally, using condition \emph{iii}, we obtain the support hypothesis, supp$(s_{\gamma} e_{-\gamma})^{\widehat{}} \subseteq \Lambda,$ of
 Theorem \ref{thm:pseudoD} for each $\gamma \in \widehat{\mathbb{R}}^d$, because of the following calculations:
 \[
 (s_{\gamma} e_{-\gamma})^{\widehat{}}(\omega) = \sum_{j=1}^{\infty} b_j(\gamma) (\widehat{a}_{j} \ast \delta_{- \lambda_j})(\omega)
 \]
 and, for each $j$,
 \[\text{supp}(\widehat{a}_j \ast \delta_{- \lambda_j}) \subseteq \overline{B(0, \epsilon_j)} + \{ \lambda_j \} \subseteq \overline{B(\lambda_j, \epsilon_j)} \subseteq \Lambda.\]
\end{example}



\section{The Beurling covering theorem}
\label{sec:covering}

Let $\Lambda \subseteq \Rhat^d$ be a convex, compact set which is symmetric about the origin and has non-empty interior. Then $\norm{\cdot}_\Lambda$, defined by
$$
   \forall \gamma \in \Rhat^d ,\quad \norm{\gamma}_\Lambda = \inf \{\rho > 0 : \gamma \in {\rho}\Lambda \},
$$
is a norm on $\Rhat^d$ equivalent to the Euclidean norm. The polar set $\Lambda^\ast \subseteq \RR^d$ of $\Lambda$ is defined as
$$
   \Lambda^\ast = \{x \in \RR^d: x \cdot \gamma \leq 1, \text{ for all } \gamma \in \Lambda \}.
$$
It is elementary to check that $\Lambda^\ast$ is a convex, compact set which is symmetric about the origin, and that it has non-empty interior.
\begin{example}
Let $\Lambda = [-1,1] \times [-1,1]$. Then, for $(\gamma_1,\gamma_2)\in \Rhat^2$,
$$
\norm{(\gamma_1,\gamma_2)}_\Lambda = \inf \{\rho >0: |\gamma_1| \leq \rho, |\gamma_2| \leq \rho \}= \norm{(\gamma_1,\gamma_2)}_\infty.
$$
The polar set of $\Lambda$ is
$$
\Lambda^\ast = \{(x_1,x_2): |x_1|+|x_2|\leq 1\} = \{(x_1,x_2): \norm{(x_1,x_2)}_1 \leq 1\}.
$$
 
\end{example}

\begin{thm}\label{thm:covering}(Beurling covering theorem)
 Let $\Lambda \subseteq \Rhat^d$ be a convex, compact set which is symmetric about the origin and has non-empty interior, and let $E \subseteq \RR^d$ be a separated set satisfying the covering property, $$\bigcup_{y \in E}\tau_y \Lambda^\ast = \RR^d.$$
If $\rho <1/4$, then $\{(e_{-x} \ \mathbb{1}_{\Lambda})^\vee: x \in E \}$ is a Fourier frame for $PW_{\rho\Lambda}$.
\end{thm}

Theorem \ref{thm:covering} \cite{BenWu1999b}, \cite{BenWu2000} involves the Paley-Wiener theorem and properties of balayage, and it depends on the theory developed in \cite{beur1989}, pages 341-350, \cite{beur1966}, and \cite{land1967}.
For a recent development, see \cite{OleUla2012}.

\section{Epilogue}\
\label{sec:epilogue}
This paper is rooted in Beurling's deep ideas and techniques dealing with balayage, that themselves have spawned
wondrous results in a host of areas ranging from Kahane's creative formulation and theory exposited in \cite{kaha1970} to the
setting of various locally compact abelian groups with surprising twists and turns and many open problems, e.g., 
\cite{shap1977}, \cite{shap1978}, to the new original chapter on quasi-crystals led by by Yves
Meyer, e.g., \cite{hof1995}, \cite{meye1995}, \cite{laga2000}, \cite{MatMey2009},
\cite{MatMey2010}, \cite{MatMeyOrt2013}, \cite{meye2012} as well as the revisiting by Beurling \cite{beur1985}. 

Even with the focused theme of this paper, there is the important issue of implementation and computation
vis a vis balayage and genuine applications of non-uniform sampling. 

\section*{Acknowledgements}
The first named author gratefully acknowledges the support of MURI-AFOSR Grant FA9550-05-1-0443.
The second named author gratefully acknowledges the support of MURI-ARO Grant W911NF-09-1-0383,
NGA Grant HM-1582-08-1-0009, and DTRA Grant HDTRA 1-13-1-0015.
Both authors also benefited from insightful observations by Professors Carlos Cabrelli, Hans
Feichtinger,
Matei Machedon, Basarab Matei, Ursula Molter, and Kasso Okoudjou. 
Further, although the interest of the second named author in this topic goes back to the 1960s,
he is especially appreciative of his collaboration in the late 1990s with Dr.~Hui-Chuan ~Wu
related to Theorem \ref{theorem:Generalized_fourier_frames} and the material in Section \ref{sec:covering}.
Finally, the second named author has had
the unbelievably good fortune through the years to learn from Henry J. Landau,
a grand master in every way. His explicit contributions for this paper are noted in Section \ref{sec:landau}.


\bibliographystyle{amsplain}
\bibliography{2013-09-26JBbib.bib}

\end{document}